\documentclass[11pt]{amsart}

\usepackage[T1]{fontenc}
\usepackage[utf8]{inputenc}
\usepackage{lmodern}

\usepackage{amsmath,amssymb,amsthm}
\usepackage{mathrsfs}

\usepackage[a4paper,margin=1.15in]{geometry}
\usepackage{microtype}
\usepackage{xcolor}

\usepackage{tikz}
\usepackage{pgfplots}
\usepgfplotslibrary{groupplots,colormaps}
\usetikzlibrary{arrows.meta,decorations.pathreplacing}
\usepackage[section]{placeins}
\pgfplotsset{compat=1.18}

\usepackage[
    colorlinks=true,
    linkcolor=blue,
    citecolor=blue,
    urlcolor=blue
]{hyperref}

\usepackage[nameinlink,capitalise,noabbrev]{cleveref}

\numberwithin{equation}{section}

\newtheorem{theorem}{Theorem}[section]
\newtheorem{lemma}[theorem]{Lemma}

\theoremstyle{definition}
\newtheorem{definition}[theorem]{Definition}

\theoremstyle{remark}
\newtheorem{remark}[theorem]{Remark}

\newcommand{\R}{\mathbb{R}}
\newcommand{\Sph}{\mathbb{S}}

\newcommand{\tr}{\operatorname{tr}}
\newcommand{\sgn}{\operatorname{sgn}}
\newcommand{\cof}{\operatorname{cof}}
\newcommand{\diag}{\operatorname{diag}}
\newcommand{\Int}{\operatorname{int}}

\newcommand{\eps}{\varepsilon}

\definecolor{SLCESeed}{HTML}{4B5563}
\definecolor{SLCEPost}{HTML}{C65D21}
\definecolor{SLCEOther}{HTML}{2F6F9F}
\definecolor{SLCESingular}{HTML}{A61B29}
\definecolor{SLCELimit}{HTML}{111827}
\definecolor{SLCESmoothA}{HTML}{D97706}
\definecolor{SLCESmoothB}{HTML}{2563A6}
\definecolor{SLCESmoothC}{HTML}{14866D}
\definecolor{SLCEGrid}{HTML}{D8DEE8}

\tikzset{
  slce point/.style={circle, fill=SLCESingular, draw=white,
    line width=0.35pt, inner sep=1.55pt},
  slce label/.style={font=\scriptsize, fill=white, fill opacity=0.92,
    text opacity=1, rounded corners=0.6pt, inner sep=1.3pt},
  slce arrow/.style={-{Stealth[length=1.8mm,width=1.25mm]},
    line width=0.55pt},
}

\pgfplotsset{
  slce axis/.style={
    tick label style={font=\scriptsize},
    label style={font=\small},
    title style={font=\small, align=center},
    axis line style={black!70, line width=0.45pt},
    tick style={black!65, line width=0.4pt},
    major grid style={draw=SLCEGrid, line width=0.25pt},
    grid=major,
    line cap=round,
    line join=round,
    clip=false,
  },
}

\pgfmathsetmacro{\figTau}{1.00}
\pgfmathsetmacro{\figAlpha}{0.75}
\pgfmathsetmacro{\figWalpha}{sqrt(1+\figAlpha*\figAlpha)}
\pgfmathsetmacro{\figRzero}{1/(\figTau*\figWalpha)}
\pgfmathsetmacro{\figCusp}{0.62}
\pgfmathsetmacro{\figLocalZmax}{0.34}
\pgfmathsetmacro{\figGraphZmax}{0.40}
\pgfmathsetmacro{\figYmax}{0.44}
\pgfmathsetmacro{\figSurfaceZmax}{0.46}

\pgfmathdeclarefunction{figrhoN}{1}{\pgfmathparse{1.35 + 0.22*(#1) - 0.06*(#1)^2}}
\pgfmathdeclarefunction{figZN}{1}{\pgfmathparse{-0.50*(#1) - 0.04*(#1)^2}}
\pgfmathdeclarefunction{figdrhoN}{1}{\pgfmathparse{0.22 - 0.12*(#1)}}
\pgfmathdeclarefunction{figdZN}{1}{\pgfmathparse{-0.50 - 0.08*(#1)}}
\pgfmathdeclarefunction{figspeedN}{1}{\pgfmathparse{sqrt((figdrhoN(#1))^2 + (figdZN(#1))^2)}}
\pgfmathdeclarefunction{fignurhoN}{1}{\pgfmathparse{-(figdZN(#1))/figspeedN(#1)}}
\pgfmathdeclarefunction{fignuZN}{1}{\pgfmathparse{figdrhoN(#1)/figspeedN(#1)}}
\pgfmathdeclarefunction{figrhoM}{1}{\pgfmathparse{figrhoN(#1) - \figTau*fignurhoN(#1)}}
\pgfmathdeclarefunction{figZM}{1}{\pgfmathparse{figZN(#1) - \figTau*fignuZN(#1)}}

\pgfmathdeclarefunction{figxPost}{1}{\pgfmathparse{\figAlpha*(#1) + \figCusp*(#1)^(1.5)}}
\pgfmathdeclarefunction{figxOther}{1}{\pgfmathparse{\figAlpha*(#1) - \figCusp*(#1)^(1.5)}}
\pgfmathdeclarefunction{figRPost}{1}{\pgfmathparse{\figRzero + \figAlpha*(#1) + \figCusp*(#1)^(1.5)}}

\pgfmathsetmacro{\figLimitC}{0.78}
\pgfmathsetmacro{\figRhoLimit}{1.15}
\pgfmathsetmacro{\figLimitZmax}{1.00}
\pgfmathsetmacro{\figLimitYmax}{0.88}
\pgfmathsetmacro{\figEpsA}{0.28}
\pgfmathsetmacro{\figEpsB}{0.12}
\pgfmathsetmacro{\figEpsC}{0.045}

\title[Counterexamples for the SLCE]
{Some Counterexamples for the Special Lagrangian Curvature Equation}

\author{Guohuan Qiu}
\address{
Institute of Mathematics,
Academy of Mathematics and Systems Science,
Chinese Academy of Sciences,
No.\ 55 Zhongguancun East Road,
Beijing 100190,
China
}
\email{qiugh@amss.ac.cn}

\author{Guanyu Tao}
\address{
Institute of Mathematics,
Academy of Mathematics and Systems Science,
Chinese Academy of Sciences,
No.\ 55 Zhongguancun East Road,
Beijing 100190,
China
}
\email{taoguanyu@amss.ac.cn}



\begin{document}

\begin{abstract}
We construct three counterexamples for the special Lagrangian curvature equation (SLCE). First, in dimension two, we use a post-focal branch of a parallel surface with constant positive Gauss curvature to construct an explicit Lipschitz viscosity solution which is not $C^1$. Second, still in dimension two, we construct a sequence of smooth admissible solutions on a fixed rectangle with uniformly bounded $C^1$-norm but unbounded curvature at one point; furthermore, we show that any uniform $C^{1,\beta}$ estimate fails for $\beta > 1/3$. Third, in dimension three and in the subcritical phase, we construct a Mooney-Savin type Lipschitz viscosity solution whose gradient has a jump discontinuity across an analytic surface.

These examples demonstrate the sharpness of the recent a priori estimates by Qiu and Zhou, revealing that their structural assumptions—convexity and the critical phase—are strictly necessary.
\end{abstract}

\maketitle

\section{Introduction}\label{sec:introduction}

Let $M^n\subset \R^{n+1}$ be an oriented hypersurface, and let
\[
    \kappa=(\kappa_1,\ldots,\kappa_n)
\]
denote its principal curvatures with respect to the chosen unit normal.
The special Lagrangian curvature equation (SLCE) is
\begin{equation}\label{eq:SLCE-intro}
    \sum_{i=1}^n \arctan \kappa_i=\Theta.
\end{equation}
It was introduced by Smith \cite{Smi13} as a curvature analogue of the
classical special Lagrangian equation originating in the work of Harvey and
Lawson \cite{HL82}.  For a graph $M=\{(x,u(x))\}$, the nonlinearity in
\eqref{eq:SLCE-intro} acts on the shape operator of the graph, rather than
directly on the Hessian $D^2u$.  Consequently, its regularity theory reflects
both the fully nonlinear structure of the equation and the extrinsic geometry
of the graph.

Interior second-derivative estimates for fully nonlinear elliptic equations
are known to depend strongly on dimension, phase, and convexity.  In dimension
two, Heinz obtained the fundamental interior estimate for the
Monge--Amp\`ere equation and the Weyl embedding problem \cite{Hei59}, whereas
Pogorelov constructed non-$C^2$ convex solutions in higher dimensions
\cite{Pog78}; see also Urbas \cite{Urb90} for related Hessian equations.
For the classical special Lagrangian equation, interior Hessian and gradient
estimates in the critical and supercritical regimes were established by
Warren--Yuan and Wang--Yuan \cite{WY10,WY14}, and the convex case was treated
by Chen--Warren--Yuan \cite{CWY09}.  Related curvature estimates for
hypersurface equations were obtained, among others, by Sheng--Urbas--Wang
\cite{SUW04}, Guan--Qiu \cite{GQ19}, Qiu \cite{Qiu24}, and
Guan--Ren--Wang \cite{GRW15}.

Qiu and Zhou initiated the interior regularity theory for the SLCE in
\cite{QZ24}.  They proved interior curvature estimates for smooth graphical
solutions in the critical phase
\[
    \Theta=\frac{(n-2)\pi}{2}
\]
and in the convex case, together with interior gradient estimates for every
constant phase.  Their argument combines the bounded-mean-curvature geometry
of the graph $(X,\nu)$, the Michael--Simon mean-value inequality
\cite{MS73}, a strong trace Jacobi inequality, and a modified integral method
in the spirit of Warren--Yuan \cite{WY09} and Wang--Yuan \cite{WY14}.
The same analysis also reveals the principal obstruction in the nonconvex
regime: away from the critical phase, a curvature term in the Jacobi inequality
may have the wrong sign.

The purpose of this paper is to show that these restrictions reflect genuine
geometric phenomena rather than limitations of the method.  We construct three
counterexamples.  The first two are two-dimensional and arise from focal
degeneration of parallel surfaces; the third is a three-dimensional
subcritical construction of Mooney--Savin type.  Together they isolate three
distinct mechanisms: a reflected first-order focal singularity, smooth
higher-order focal concentration, and Bellman--Legendre collapse.

We begin with the two-dimensional geometry.  With the convention
\[
    H=\kappa_1+\kappa_2,
    \qquad
    K=\kappa_1\kappa_2,
\]
the SLCE becomes
\[
    \cos\Theta\,H-\sin\Theta\,(1-K)=0.
\]
For
\[
    0<\Theta<\frac{\pi}{2},
    \qquad
    \tau=\cot\Theta>0,
\]
the phase branch considered in this paper is equivalently described by
\begin{equation}\label{eq:linear-Weingarten-intro}
    K+\tau H=1,
\end{equation}
together with
\begin{equation}\label{eq:admissible-branch-intro}
    \kappa_1+\tau>0,
    \qquad
    \kappa_2+\tau>0.
\end{equation}
Indeed,
\[
    K+\tau H=1
    \quad\Longleftrightarrow\quad
    (\kappa_1+\tau)(\kappa_2+\tau)=1+\tau^2,
\]
and \eqref{eq:admissible-branch-intro} selects the required branch of the
arctangent equation.  Here ``admissible'' refers specifically to this
phase-branch condition; it is not a convexity assumption.

The form \eqref{eq:linear-Weingarten-intro} makes parallel surfaces natural.
Let $\Sigma_0\subset\R^3$ be an oriented surface with position vector $X$,
unit normal $\nu_0$, and principal curvatures $\lambda_1,\lambda_2$.  For the
signed parallel surface
\[
    \Sigma_t=X+t\nu_0,
\]
one has, wherever $\Sigma_t$ is regular,
\[
    \lambda_i(t)=\frac{\lambda_i}{1-t\lambda_i},
    \qquad i=1,2.
\]
If $\Sigma_0$ has constant positive Gauss curvature
\[
    K_0=\frac{1}{1+\tau^2},
\]
then the parallel surface at signed distance $t=-\tau$ satisfies
\eqref{eq:linear-Weingarten-intro} wherever it is regular.  Its singular set is
the focal set determined by
\[
    1+\tau\lambda_i=0.
\]
Thus the same Jacobi factor that causes the parallel map to lose rank also
forces a principal curvature of the SLCE surface to blow up.

This construction belongs to the classical geometry of parallel surfaces and
the modern theory of fronts.  The relation between constant positive Gauss
curvature and constant mean curvature parallel surfaces goes back to Bonnet;
see \cite{GMM03,Dar90}.  Focal sets and singularities of the normal map were
studied by Porteous \cite{Por71} and Bruce--Wilkinson \cite{BW91}, while
Saji--Umehara--Yamada \cite{SUY09} and Fukui--Hasegawa \cite{FH12}
developed local criteria for singularities of fronts and parallel surfaces.
For the present paper, the relevant point is that a transverse focal crossing
produces a cuspidal-edge-type local model, whereas a higher-order vanishing of
the Jacobi factor produces a smoother but curvature-concentrating degeneration.

There is also a useful connection with optimal transport.  Qiu and Zhou
observed in \cite[Section~6]{QZ24} that the graphical form of
\eqref{eq:linear-Weingarten-intro} is related to a Monge--Amp\`ere type
transport equation with relativistic cost \cite{Bre03,BP13}
\[
    c(x,y)=-\sqrt{\tan^2\Theta-|x-y|^2}
\]
and densities
\[
    f=\frac{1}{\cos^2\Theta},
    \qquad g=1.
\]
Equivalently, one may consider the signed family
\[
    c_\tau(x,y)
    =-\sgn(\tau)\sqrt{\tau^{-2}-|x-y|^2},
    \qquad \tau\neq0.
\]
The associated Ma--Trudinger--Wang tensor \cite{MTW05} has the unfavorable
sign when $\tau>0$, while the sign is favorable for $\tau\leq0$ (with
$\tau=0$ understood as a limit).  This analogy is informative but does not
itself produce our examples: Loeper's construction \cite{Loe09} uses freedom
in the source and target densities and corresponds, in the present variables,
to an equation of the form
\[
    K+\tau H=F(x,u,\nabla u),
\]
rather than the constant equation \eqref{eq:linear-Weingarten-intro}.  Our
counterexamples therefore concern the SLCE itself, with fixed right-hand side.

Our first result is a geometrically explicit Lipschitz singular solution in 
dimension two. We exploit the focal set of parallel surfaces. Starting from a 
rotational surface of constant positive Gauss curvature, we take its signed parallel 
surface, choose a non-symmetric focal point, and retain the admissible post-focal branch. 
By reflecting this branch across the singular line, we obtain a local graph with an 
exact first-order corner, which is Lipschitz but not $C^1$. A direct test-function 
argument along this singular line then verifies the viscosity inequalities for the 
constant SLCE.  

\begin{theorem}\label{thm:two-dimensional-Lipschitz}
For every
\[
    0<\Theta<\frac{\pi}{2}
\]
and every $\alpha>0$, there exist a rectangle $Q\subset\R^2$ and a viscosity
solution
\[
    u\in C^{0,1}(Q)\setminus C^1(Q)
\]
of
\[
    \arctan\kappa_1+\arctan\kappa_2=\Theta
    \qquad\text{in }Q.
\]
Moreover, $u$ is a classical solution in $Q\setminus\{z=0\}$ and, as
$(y,z)\to(0,0)$,
\[
    u(y,z)
    =\alpha|z|
     -\frac{\tau\sqrt{1+\alpha^2}}{2}y^2
     +O\bigl(|z|^{3/2}+|y|^2|z|+|y|^4\bigr),
\]
where $\tau=\cot\Theta$.
\end{theorem}

Our second result reveals that smooth higher-order focal concentration essentially 
corresponds to the degenerate case $\alpha=0$ in the preceding Lipschitz construction. 
Because $\alpha=0$, the parallel surface no longer forms a corner requiring even 
reflection; rather, it naturally maintains a smooth connection at the focal point. 
Without needing reflection, we directly estimate the regularity near the focal point 
by analyzing a one-parameter family of smooth admissible branches approaching this 
degeneration. This family ultimately converges to a limiting profile with a precise 
$\vert{}z\vert{}^{4/3}$ focal scaling. This effectively identifies $C^{1,1/3}$ as the 
exact critical regularity threshold for nonconvex solutions.

\begin{theorem}\label{thm:two-dimensional-C2-failure}
For every
\[
    0<\Theta<\frac{\pi}{2},
\]
there exist a rectangle $Q\subset\R^2$ containing the origin and a sequence
$u_\eps$ of functions, smooth on a neighborhood of $Q$, whose
graphs solve
\[
    \arctan\kappa_1+\arctan\kappa_2=\Theta
    \qquad\text{in }Q
\]
on the branch \eqref{eq:admissible-branch-intro}, such that
\[
    \sup_\eps\|u_\eps\|_{C^{1,1/3}(Q)}<\infty,
    \qquad
    |D^2u_\eps(0)|\longrightarrow\infty.
\]
For every $\beta\in(1/3,1]$,
\[
    [Du_\eps]_{C^{0,\beta}(Q)}\longrightarrow\infty.
\]
Moreover, $u_\eps$ converges uniformly to a limiting profile $u_0$ with the
exact focal asymptotic
\[
    u_0(0,z)
    =\frac{3c_\Theta}{4}|z|^{4/3}+O(z^2),
    \qquad
    c_\Theta=\left(\frac{6}{\tau K_0}\right)^{1/3}.
\]
In particular, $u_0$ is $C^{1,1/3}$ but is not $C^2$ in the focal direction.
\end{theorem}

The solutions in \cref{thm:two-dimensional-C2-failure} are necessarily
nonconvex.  Adding flat variables gives analogous smooth examples in every
dimension $n\geq2$: the new principal curvatures are zero, so the numerical
phase is unchanged, although its position relative to the
higher-dimensional critical threshold changes.

Our third result concerns the subcritical phase in dimension three.  For the
classical special Lagrangian equation, Nadirashvili--Vl\u{a}du\c{t}
\cite{NV10} and Wang--Yuan \cite{WY13} constructed $C^{1,\alpha}$ singular
solutions below the critical phase, and Mooney--Savin \cite{MS24} constructed
viscosity solutions that are Lipschitz but not $C^1$.  We obtain the analogous
failure of $C^1$ regularization for the three-dimensional SLCE.

\begin{theorem}\label{thm:three-dimensional-Lipschitz}
For every
\[
    0<\Theta<\frac{\pi}{2},
\]
there exist a smooth bounded domain $\Omega\subset\R^3$, a compact
real-analytic embedded surface with boundary $\Gamma\Subset\Omega$, and a
viscosity solution
\[
    u\in C^{0,1}(\Omega)\setminus C^1(\Omega)
\]
of
\[
    \sum_{i=1}^3\arctan\kappa_i=\Theta
    \qquad\text{in }\Omega.
\]
Moreover, $u$ is a classical solution in $\Omega\setminus\Gamma$, and $Du$
has a jump discontinuity across $\Gamma\setminus\partial\Gamma$.
\end{theorem}

For the subcritical phase in dimension three, our construction adapts the 
Bellman-Legendre collapse mechanism of Mooney-Savin. The core distinction is 
that our SLCE dual operator $\mathcal{F}(x, D^2w)$ depends explicitly on the spatial 
coordinate $x$ via congruence by the positive definite matrix $P(x)^{1/2}$. 
However, because the singularity is constructed near the origin, $P(0)^{1/2}$ is 
exactly the identity matrix. Thus, our operator is effectively the Mooney-Savin 
operator subjected to a small perturbation around the identity, which our calculations 
show does not destabilize the overall structural framework. To execute this 
systematically, we construct a strictly convex phase core where the Hessian has rank 
two, solve an analytic Cauchy problem across its boundary, and demonstrate a strict 
determinant sign change in a uniform exterior collar. The explicit spatial dependence 
cancels perfectly when comparing normal derivatives on the boundary because the inner 
and outer Hessians match exactly. This yields a $C^{2,1}$ dual potential whose gradient 
map collapses one-dimensional fibres onto an analytic surface. A local Legendre 
transform then converts this geometric collapse into a genuine gradient jump across 
the surface, with smooth strict subsolutions constructed to secure the viscosity 
inequalities on the singular set.

The two-dimensional examples and the three-dimensional example thus exhibit
different obstructions.  Focal degeneration prevents a curvature estimate from
following merely from positive phase and uniform $C^1$ control on the
nonconvex branch, while Bellman--Legendre collapse shows that subcritical
Lipschitz viscosity solutions need not become $C^1$.  In this sense, the phase
threshold and the convexity hypothesis in the known interior theory are
substantive features of the SLCE rather than technical artifacts.

\section{Preliminaries}\label{sec:preliminaries}

In this section, we introduce the graphical form of the special
Lagrangian curvature equation and fix the viscosity convention.  We
then record the Legendre-transform identity used in the
three-dimensional construction and the two-dimensional graphical
formulas needed below.

\subsection{The graphical equation and viscosity solutions}
\label{subsec:graphical-equation}

Let $\Omega\subset\R^n$, and let
\[
M=\{(x,u(x)):x\in\Omega\}\subset\R^{n+1}
\]
be the graph of a function $u\in C^2(\Omega)$. With respect to the
upward unit normal
\[
\nu=\frac{(-Du,1)}{\sqrt{1+|Du|^{2}}},
\]
the induced metric and the second fundamental form are
\[
g=I+Du\otimes Du,
\qquad
h=\frac{D^{2}u}{\sqrt{1+|Du|^{2}}}.
\]
Consequently, the shape operator is
\[
A[u]
=
g^{-1}h
=
\left(
I-\frac{Du\otimes Du}{1+|Du|^{2}}
\right)
\frac{D^{2}u}{\sqrt{1+|Du|^{2}}}.
\]

For $p\in\R^n$, define
\[
P(p)
:=
\sqrt{1+|p|^{2}}\,(I+p\otimes p).
\]
Then $P(p)$ is positive definite, and $A[u]$ is similar to the
symmetric matrix
\[
S[u]
:=
P(Du)^{-1/2}D^{2}u\,P(Du)^{-1/2}.
\]
In particular, the eigenvalues of $S[u]$ are the principal
curvatures of the graph. Therefore, the graphical SLCE is
\begin{equation}
\label{eq:graphical-SLCE}
\tr\bigl(\arctan S[u]\bigr)=\Theta.
\end{equation}
Here, for a symmetric matrix, the matrix arctangent is defined by
spectral calculus.

Let $\R_{\mathrm{sym}}^{n \times n}$ denote the space of $n \times n$ real symmetric matrices. For $p\in\R^n$ and
$X\in\R_{\mathrm{sym}}^{n\times n}$, set
\[
\begin{aligned}
F_{\Theta}(p,X)
&:=
\tr\!\left[
\arctan\!\left(
P(p)^{-1/2}XP(p)^{-1/2}
\right)
\right]-\Theta,
\\
\widetilde{F}_{\Theta}(p,X)
&:=
\tr\!\left[
\arctan\!\left(
P(p)^{1/2}XP(p)^{1/2}
\right)
\right]-\Theta.
\end{aligned}
\]
Thus, equation \eqref{eq:graphical-SLCE} is equivalent to
\[
F_{\Theta}(Du,D^{2}u)=0.
\]

\begin{lemma}\label{lem:operator-monotonicity}
For every fixed $p\in\R^{n}$, both operators
\[
X\longmapsto F_{\Theta}(p,X)
\qquad\text{and}\qquad
X\longmapsto\widetilde{F}_{\Theta}(p,X)
\]
are nondecreasing. More precisely, if
\[
Y-X\geq 0,
\]
then
\[
F_{\Theta}(p,Y)\geq F_{\Theta}(p,X)
\]
and
\[
\widetilde{F}_{\Theta}(p,Y)
\geq
\widetilde{F}_{\Theta}(p,X).
\]
\end{lemma}

\begin{proof}
For $F_{\Theta}$, set
\[
\widehat{X}
:=
P(p)^{-1/2}XP(p)^{-1/2},
\qquad
\widehat{Y}
:=
P(p)^{-1/2}YP(p)^{-1/2}.
\]
The inequality $Y-X\geq 0$ implies
\[
\widehat{Y}-\widehat{X}\geq 0.
\]
Let
\[
M_{t}
:=
\widehat{X}
+t(\widehat{Y}-\widehat{X}),
\qquad
0\leq t\leq 1.
\]
Then
\[
\frac{d}{dt}
\tr\bigl(\arctan M_{t}\bigr)
=
\tr\!\left[
(I+M_{t}^{2})^{-1}
(\widehat{Y}-\widehat{X})
\right]
\geq 0.
\]
Hence
\[
F_{\Theta}(p,Y)\geq F_{\Theta}(p,X).
\]
The proof for $\widetilde{F}_{\Theta}$ is identical, with
$P(p)^{-1/2}$ replaced by $P(p)^{1/2}$.
\end{proof}

Because $F_{\Theta}$ is nondecreasing in the Hessian variable, we
use the following viscosity convention.

\begin{definition}
\label{def:viscosity-solution}
Let $u\in C^{0}(\Omega)$.

The function $u$ is a viscosity subsolution of the SLCE if, whenever
$\varphi\in C^{2}(\Omega)$ touches $u$ from above at
$x_{0}\in\Omega$, one has
\[
F_{\Theta}
\bigl(
D\varphi(x_{0}),
D^{2}\varphi(x_{0})
\bigr)
\geq 0.
\]

The function $u$ is a viscosity supersolution of the SLCE if,
whenever $\varphi\in C^{2}(\Omega)$ touches $u$ from below at
$x_{0}\in\Omega$, one has
\[
F_{\Theta}
\bigl(
D\varphi(x_{0}),
D^{2}\varphi(x_{0})
\bigr)
\leq 0.
\]

A viscosity solution is both a viscosity subsolution and a viscosity
supersolution.
\end{definition}

\begin{remark}\label{rem:viscosity-sign-convention}
The directions of the inequalities in
Definition~\ref{def:viscosity-solution} reflect the fact that
$F_{\Theta}$ is nondecreasing, rather than nonincreasing, in its
Hessian variable. If no $C^{2}$ test function touches from the
relevant side at a given point, then the corresponding viscosity
condition is vacuous.
\end{remark}

\subsection{Legendre duality in dimension three}
\label{subsec:Legendre-duality}

Let $w\in C^{2}(U)$, where $U\subset\R^{3}$, and suppose
that its gradient map is locally one-to-one. The local Legendre
transform of $w$ is defined by
\[
u(y)=x\cdot y-w(x),
\qquad
y=Dw(x).
\]
Whenever $D^{2}w(x)$ is nonsingular, one has
\begin{equation}
\label{eq:Legendre-derivatives}
Du(y)=x,
\qquad
D^{2}u(y)=\bigl(D^{2}w(x)\bigr)^{-1}.
\end{equation}

For a function $\psi\in C^{2}(U)$, define
\[
\widetilde{S}[\psi](x)
:=
P(x)^{1/2}D^{2}\psi(x)\,P(x)^{1/2}.
\]
Thus,
\[
\widetilde{F}_{\Theta}(x,D^{2}\psi)
=
\tr
\bigl(
\arctan\widetilde{S}[\psi](x)
\bigr)
-\Theta.
\]

\begin{lemma}
\label{lem:Legendre-transform-identity}
Let $w$ be of class $C^{2}$ in a neighborhood of a point
$x\in\R^{3}$. Assume that $D^{2}w(x)$ is nonsingular and
has exactly two positive eigenvalues and one negative eigenvalue.
Suppose also that $Dw$ is locally one-to-one near $x$, and let
$u=w^{*}$ be the local Legendre transform near
\[
y=Dw(x).
\]
Then
\begin{equation}
\label{eq:Legendre-phase-identity}
\tr
\bigl(
\arctan S[u](y)
\bigr)
=
\frac{\pi}{2}
-
\tr
\bigl(
\arctan\widetilde{S}[w](x)
\bigr).
\end{equation}
Consequently, if
\[
\widetilde{F}_{\Theta^{*}}
\bigl(
x,D^{2}w(x)
\bigr)
=0
\]
and
\[
\Theta=\frac{\pi}{2}-\Theta^{*},
\]
then $u$ is a classical solution of the SLCE with phase $\Theta$
at $y$.
\end{lemma}

\begin{proof}
By \eqref{eq:Legendre-derivatives},
\[
Du(y)=x,
\qquad
D^{2}u(y)=\bigl(D^{2}w(x)\bigr)^{-1}.
\]
Therefore,
\[
\begin{aligned}
S[u](y)
&=
P(x)^{-1/2}
\bigl(D^{2}w(x)\bigr)^{-1}
P(x)^{-1/2}
\\
&=
\left(
P(x)^{1/2}
D^{2}w(x)
P(x)^{1/2}
\right)^{-1}
\\
&=
\widetilde{S}[w](x)^{-1}.
\end{aligned}
\]
Let $\mu_{1},\mu_{2},\mu_{3}$ be the eigenvalues of
$\widetilde{S}[w](x)$. Since congruence by the positive definite
matrix $P(x)^{1/2}$ preserves inertia, exactly two of the
$\mu_{i}$ are positive and one is negative. For every
$\mu\neq 0$,
\[
\arctan(\mu^{-1})
=
\sgn(\mu)\frac{\pi}{2}
-\arctan\mu.
\]
Summing over the three eigenvalues gives
\[
\sum_{i=1}^{3}\arctan(\mu_{i}^{-1})
=
\frac{\pi}{2}
-
\sum_{i=1}^{3}\arctan\mu_{i},
\]
which proves \eqref{eq:Legendre-phase-identity}.
\end{proof}

\subsection{The two-dimensional equation}
\label{subsec:two-dimensional-equation}

In the two-dimensional constructions, local graphs are written as
\[
x=u(y,z)
\]
in the ambient coordinates $(x,y,z)\in\R^{3}$ and are oriented by
\begin{equation}
\label{eq:two-dimensional-graph-normal}
\nu_{u}
=
\frac{(1,-u_{y},-u_{z})}
{\sqrt{1+u_{y}^{2}+u_{z}^{2}}}.
\end{equation}
The induced metric and second fundamental form are
\[
g
=
\begin{pmatrix}
1+u_{y}^{2} & u_{y}u_{z} \\
u_{y}u_{z} & 1+u_{z}^{2}
\end{pmatrix},
\qquad
h
=
\frac{1}{\sqrt{1+|Du|^{2}}}
\begin{pmatrix}
u_{yy} & u_{yz} \\
u_{yz} & u_{zz}
\end{pmatrix}.
\]
In particular, the normal curvature in the coordinate direction
$\partial_{y}$ is
\begin{equation}
\label{eq:y-normal-curvature}
k_{y}
=
\frac{u_{yy}}
{\sqrt{1+|Du|^{2}}\,(1+u_{y}^{2})}.
\end{equation}
If $\kappa_{1}\leq\kappa_{2}$ are the principal curvatures, then
\begin{equation}
\label{eq:normal-curvature-comparison}
\kappa_{1}\leq k_{y}\leq\kappa_{2}.
\end{equation}

For a surface in $\R^{3}$, write
\[
H:=\kappa_{1}+\kappa_{2},
\qquad
K:=\kappa_{1}\kappa_{2}.
\]
For
\[
0<\Theta<\frac{\pi}{2},
\qquad
\tau:=\cot\Theta>0,
\]
the relevant branch of the two-dimensional SLCE is equivalently
\[
K+\tau H=1.
\]
We record the branch condition explicitly.

\begin{lemma}
\label{lem:two-dimensional-equivalence}
Let $0<\Theta<\pi/2$ and $\tau=\cot\Theta$. Under
\begin{equation}
\label{eq:two-dimensional-admissible-branch}
\kappa_{1}+\tau>0,
\qquad
\kappa_{2}+\tau>0,
\end{equation}
one has
\[
\arctan\kappa_{1}+\arctan\kappa_{2}=\Theta
\]
if and only if
\[
K+\tau H=1.
\]
\end{lemma}

\begin{proof}
The tangent addition formula gives the forward implication. Conversely,
$K+\tau H=1$ and \eqref{eq:two-dimensional-admissible-branch} imply
\[
(\kappa_{1}+\tau)(\kappa_{2}+\tau)=1+\tau^{2},
\]
so $H>0$ and $1-K=\tau H>0$. Hence
\[
\widetilde\Theta
:=
\arctan\kappa_{1}+\arctan\kappa_{2}
\in\left(0,\frac{\pi}{2}\right),
\qquad
\tan\widetilde\Theta
=
\frac{H}{1-K}
=
\frac1\tau
=
\tan\Theta,
\]
and therefore $\widetilde\Theta=\Theta$.
\end{proof}

\section{A two-dimensional Lipschitz viscosity solution}
\label{sec:two-dimensional-Lipschitz}

We prove Theorem~\ref{thm:two-dimensional-Lipschitz} by choosing a
rotational surface of constant Gauss curvature whose meridian focal
factor vanishes simply at one point.  We retain the admissible
post-focal side of its signed parallel surface, write it as a one-sided
graph, and reflect it across the focal line.

\begin{proof}[Proof of
Theorem~\ref{thm:two-dimensional-Lipschitz}]
Fix
\[
0<\Theta<\frac{\pi}{2},
\qquad
\alpha>0,
\qquad
\tau:=\cot\Theta,
\qquad
K_{0}:=\frac{1}{1+\tau^{2}}.
\]
For every function $f=f(s)$ introduced below, write $f_{0}:=f(0)$.
We shrink $s_{0}>0$ whenever necessary, without changing the notation;
a simultaneous choice is given in
Remark~\ref{rem:two-dimensional-Lipschitz-bookkeeping}.

Define
\[
r(s)
:=
\frac{1}{\tau K_{0}\sqrt{1+\alpha^{2}}}
\cos\bigl(\sqrt{K_{0}}s\bigr)
+
\frac{\alpha}{\sqrt{K_{0}}\sqrt{1+\alpha^{2}}}
\sin\bigl(\sqrt{K_{0}}s\bigr).
\]
Then $r''+K_{0}r=0$ and
\[
r_{0}
=
\frac{1}{\tau K_{0}\sqrt{1+\alpha^{2}}},
\qquad
r'_{0}
=
\frac{\alpha}{\sqrt{1+\alpha^{2}}},
\qquad
r''_{0}
=
-\frac{1}{\tau\sqrt{1+\alpha^{2}}}.
\]
Since $r$ is a trigonometric polynomial, there are constants
$C_{1},C_{2}>0$ such that
\[
|r'|\leq C_{1},
\qquad
|r''|\leq C_{2}.
\]
After shrinking $s_{0}$,
\[
\frac{r_{0}}{2}\leq r\leq\frac{3r_{0}}{2},
\qquad
\frac{r'_{0}}{2}\leq r'\leq\frac{1+r'_{0}}{2}<1
\qquad (|s|<s_{0}).
\]
Set
\[
\zeta(s)
:=-\int_{0}^{s}\sqrt{1-r'(t)^{2}}\,dt,
\qquad
\zeta'=-\sqrt{1-r'^{2}}.
\]
The upper bound for $r'$ implies
\[
-\zeta'
\geq
\sqrt{1-\left(\frac{1+r'_{0}}{2}\right)^{2}}
\geq
\frac{1}{2\sqrt{1+\alpha^{2}}}.
\]
Moreover,
\[
\zeta_{0}=0,
\qquad
\zeta'_{0}=-\frac{1}{\sqrt{1+\alpha^{2}}},
\qquad
r'^{2}+\zeta'^{2}=1,
\]
and
\[
\zeta''
=-\frac{r'r''}{\zeta'}
=
\frac{K_{0}rr'}{\zeta'}.
\]
Thus $|\zeta''|\leq C_{3}$.  Shrinking $s_{0}$ again gives
\begin{equation}
\label{eq:3-basic-bounds}
\frac{r_{0}}{2}\leq r\leq\frac{3r_{0}}{2},
\qquad
\frac{r'_{0}}{2}\leq r'\leq\frac{1+r'_{0}}{2},
\qquad
\frac{1}{2\sqrt{1+\alpha^{2}}}
\leq-\zeta'
\leq
\frac{3}{2\sqrt{1+\alpha^{2}}}.
\end{equation}

Consider
\[
N(s,\theta)
=
\bigl(r(s)\cos\theta,r(s)\sin\theta,\zeta(s)\bigr),
\qquad
\nu=-\zeta'e_{r}+r'e_{z}.
\]
A direct computation gives the principal curvatures
\[
\lambda_{s}
=-\frac{r''}{\zeta'}
=
\frac{K_{0}r}{\zeta'},
\qquad
\lambda_{\theta}
=
\frac{\zeta'}{r},
\]
so $\lambda_{s}\lambda_{\theta}=K_{0}$.

For the signed parallel surface $M:=N-\tau\nu$, set
\[
J_{s}
:=1+\tau\lambda_{s}
=1+\tau\frac{K_{0}r}{\zeta'},
\qquad
J_{\theta}
:=1+\tau\lambda_{\theta}
=1+\tau\frac{\zeta'}{r}.
\]
At $s=0$,
\[
(\lambda_{s})_{0}=-\frac1\tau,
\qquad
(J_{s})_{0}=0,
\qquad
(J_{\theta})_{0}=1-\tau^{2}K_{0}=K_{0}.
\]
The conserved quantity for $r''+K_{0}r=0$ gives
\[
(r')^{2}+K_{0}r^{2}
\equiv
1+\frac{1}{\tau^{2}(1+\alpha^{2})},
\qquad
(\zeta')^{2}-K_{0}r^{2}
=
-\frac{1}{\tau^{2}(1+\alpha^{2})}.
\]
Consequently,
\[
J_{s}'
=
-\frac{K_{0}r'}
{\tau(1+\alpha^{2})(\zeta')^{3}},
\qquad
J_{s}''
=
\frac{K_{0}^{2}r\bigl(1+2(r')^{2}\bigr)}
{\tau(1+\alpha^{2})(\zeta')^{5}}.
\]
In particular,
\[
(J_{s}')_{0}=\frac{\alpha K_{0}}{\tau},
\qquad
(J_{s}'')_{0}
=-\frac{K_{0}(1+3\alpha^{2})}{\tau^{2}}.
\]
The bounds in \eqref{eq:3-basic-bounds} give
$|J_{s}''|\leq C_{4}$.  Hence, after another shrinking,
\[
\frac{\alpha K_{0}}{2\tau}
\leq J_{s}'
\leq
\frac{3\alpha K_{0}}{2\tau},
\]
and therefore
\begin{equation}
\label{eq:3-Js-linear-bounds}
\frac{\alpha K_{0}}{2\tau}(-s)
\leq-J_{s}(s)
\leq
\frac{3\alpha K_{0}}{2\tau}(-s),
\qquad -s_{0}<s<0.
\end{equation}
In particular, $J_{s}<0$ on $(-s_{0},0)$.  The same conserved
quantity also gives
\[
J_{\theta}'
=
\frac{r'}
{\tau(1+\alpha^{2})\zeta' r^{2}},
\]
so $|J_{\theta}'|\leq C_{5}$.  Shrinking $s_{0}$ once more,
\[
\frac{K_{0}}{2}
\leq J_{\theta}\leq
\frac{3K_{0}}{2}.
\]
Thus $M$ is regular on $(-s_{0},0)\times\Sph^{1}$.

Direct differentiation in the two principal directions gives
\[
\kappa_{s}=\frac{\lambda_{s}}{J_{s}},
\qquad
\kappa_{\theta}=\frac{\lambda_{\theta}}{J_{\theta}}.
\]
Since $\lambda_{s}\lambda_{\theta}=K_{0}=(1+\tau^{2})^{-1}$,
direct substitution yields
\[
K_{M}+\tau H_{M}=1.
\]
Furthermore,
\[
\kappa_{s}+\tau
=
\frac{(1+\tau^{2})J_{s}-1}{\tau J_{s}}>0,
\qquad
\kappa_{\theta}+\tau
=
\frac{\tau+\lambda_{s}^{-1}}{J_{\theta}}>0.
\]
Indeed, $J_{s}<0$ implies $\lambda_{s}<-1/\tau$, hence
$\lambda_{s}^{-1}>-\tau$.  Lemma~\ref{lem:two-dimensional-equivalence}
therefore gives
\[
\arctan\kappa_{s}+\arctan\kappa_{\theta}=\Theta
\]
on the selected post-focal branch.

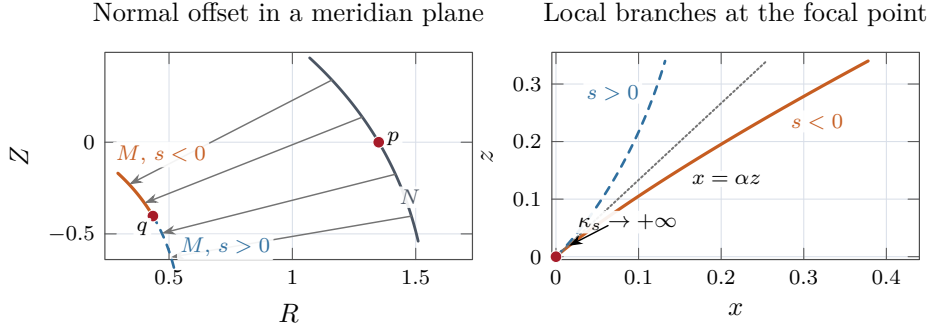
\begin{figure}[t]
\centering
\begin{tikzpicture}
\begin{groupplot}[
  group style={group size=2 by 1, horizontal sep=1.05cm},
  width=0.425\textwidth,
  height=0.285\textwidth,
  slce axis
]

\nextgroupplot[
  title={Normal offset in a meridian plane},
  xlabel={$R$}, ylabel={$Z$},
  xmin=0.24, xmax=1.73,
  ymin=-0.64, ymax=0.49
]
\addplot[SLCESeed, line width=1.05pt, domain=-1:1, samples=150]
  ({figrhoN(x)},{figZN(x)});
\addplot[SLCEPost, line width=1.2pt, domain=-1:0, samples=100]
  ({figrhoM(x)},{figZM(x)});
\addplot[SLCEOther, line width=1.0pt, dashed, domain=0:1, samples=100]
  ({figrhoM(x)},{figZM(x)});

\pgfplotsinvokeforeach{-0.72,-0.28,0.34,0.76}{%
  \draw[slce arrow, black!55]
    (axis cs:{figrhoN(#1)},{figZN(#1)}) --
    (axis cs:{figrhoM(#1)},{figZM(#1)});
}
\pgfmathsetmacro{\figrP}{figrhoN(0)}
\pgfmathsetmacro{\figzP}{figZN(0)}
\pgfmathsetmacro{\figrQ}{figrhoM(0)}
\pgfmathsetmacro{\figzQ}{figZM(0)}
\node[slce point] at (axis cs:\figrP,\figzP) {};
\node[slce point] at (axis cs:\figrQ,\figzQ) {};
\node[slce label, anchor=west] at (axis cs:1.37,0.02) {$p$};
\node[slce label, anchor=north east] at (axis cs:0.43,-0.41) {$q$};
\node[slce label, text=SLCESeed, anchor=west]
  at (axis cs:1.42,-0.30) {$N$};
\node[slce label, text=SLCEPost, anchor=south west]
  at (axis cs:0.27,-0.14) {$M$, $s<0$};
\node[slce label, text=SLCEOther, anchor=north west]
  at (axis cs:0.53,-0.50) {$M$, $s>0$};

\nextgroupplot[
  title={Local branches at the focal point},
  xlabel={$x$}, ylabel={$z$},
  xmin=-0.005, xmax=0.44,
  ymin=-0.005, ymax=0.355
]
\addplot[SLCEPost, line width=1.2pt, domain=0:\figLocalZmax, samples=180]
  ({figxPost(x)},{x});
\addplot[SLCEOther, line width=1.0pt, dashed, domain=0:\figLocalZmax, samples=180]
  ({figxOther(x)},{x});
\addplot[black!55, line width=0.75pt, densely dotted,
  domain=0:\figLocalZmax, samples=2]
  ({\figAlpha*x},{x});
\node[slce point] at (axis cs:0,0) {};
\node[slce label, text=SLCEPost, anchor=west]
  at (axis cs:0.28,0.235) {$s<0$};
\node[slce label, text=SLCEOther, anchor=east]
  at (axis cs:0.105,0.285) {$s>0$};
\node[slce label, anchor=west]
  at (axis cs:0.16,0.135) {$x=\alpha z$};
\node[slce label, anchor=west] at (axis cs:0.022,0.060)
  {$\kappa_s\to+\infty$};
\draw[slce arrow]
  (axis cs:0.063,0.054) -- (axis cs:{figxPost(0.019)},0.019);

\end{groupplot}
\end{tikzpicture}
\caption{Geometry of the two-dimensional post-focal construction.
The left panel is a schematic meridian-plane picture of the normal
shift $M=N-\tau\nu$.  The right panel displays the two local branches
with their common tangent.  The solid branch is the admissible
post-focal branch retained in the proof; the curves are drawn from the
local asymptotic model and are not to scale.}
\label{fig:postfocal-geometry}
\end{figure}

Write
\[
M(s,\theta)
=
\bigl(R(s)\cos\theta,R(s)\sin\theta,Z(s)\bigr),
\]
where
\[
R:=r+\tau\zeta',
\qquad
Z:=\zeta-\tau r'.
\]
At the focal point,
\[
R_{0}
=
\frac{1}{\tau\sqrt{1+\alpha^{2}}}>0,
\qquad
Z_{0}
=
-\frac{\tau\alpha}{\sqrt{1+\alpha^{2}}}.
\]
Using $r''=-K_{0}r$ and
$\zeta''=K_{0}rr'/\zeta'$, we obtain
\begin{equation}
\label{eq:3-RZ-prime}
R'=r'J_{s},
\qquad
Z'=\zeta'J_{s}.
\end{equation}
The first identity gives $|R'|\leq C_{6}$.  After shrinking $s_{0}$,
\begin{equation}
\label{eq:3-R-positive}
\frac{R_{0}}{2}\leq R\leq\frac{3R_{0}}{2}.
\end{equation}
Also,
\[
\left(-\frac{r'}{\zeta'}\right)_{0}=\alpha,
\qquad
\left(-\frac{r'}{\zeta'}\right)'
=
\frac{K_{0}r}{(\zeta')^{3}}.
\]
The derivative is bounded by a constant $C_{7}$, so a final shrinking
gives
\[
\frac{\alpha}{2}
\leq-\frac{r'}{\zeta'}
\leq\frac{3\alpha}{2}.
\]

Center the coordinates at the focal point by
\[
x:=R(s)\cos\theta-R_{0},
\qquad
y:=R(s)\sin\theta,
\qquad
z:=Z_{0}-Z(s).
\]
For $-s_{0}<s<0$, \eqref{eq:3-RZ-prime},
\eqref{eq:3-basic-bounds}, and
\eqref{eq:3-Js-linear-bounds} give
\[
C_{8}(-s)
\leq
-\frac{d}{ds}\bigl(Z_{0}-Z(s)\bigr)
\leq
C_{9}(-s),
\]
where one may take
\[
C_{8}
:=
\frac{\alpha K_{0}}
{4\tau\sqrt{1+\alpha^{2}}},
\qquad
C_{9}
:=
\frac{9\alpha K_{0}}
{4\tau\sqrt{1+\alpha^{2}}}.
\]
Thus
\begin{equation}
\label{eq:3-z-quadratic-bounds}
\frac{C_{8}}{2}s^{2}
\leq Z_{0}-Z(s)
\leq\frac{C_{9}}{2}s^{2},
\qquad -s_{0}<s<0.
\end{equation}
Set
\[
z_{0}:=\frac{C_{8}}{4}s_{0}^{2}.
\]
Then $s\mapsto Z_{0}-Z(s)$ has an inverse $s=s(z)<0$ for
$0<z<z_{0}$, and \eqref{eq:3-z-quadratic-bounds} yields
\[
|s(z)|\leq C_{10}z^{1/2},
\qquad
C_{10}:=\left(\frac{2}{C_{8}}\right)^{1/2}.
\]
Writing $R(z):=R(s(z))$, the identities in
\eqref{eq:3-RZ-prime} give
\[
R'(z)
=-\frac{r'(s(z))}{\zeta'(s(z))}.
\]
Hence
\[
|R'(z)-\alpha|
\leq C_{11}z^{1/2},
\qquad
C_{11}:=C_{7}C_{10},
\]
and therefore
\begin{equation}
\label{eq:3-R-expansion}
R(z)=R_{0}+\alpha z+O(z^{3/2})
\qquad (z\downarrow0).
\end{equation}

Let $y_{0}:=R_{0}/4$.  By \eqref{eq:3-R-positive},
\[
R(z)^{2}-y^{2}
\geq
\frac{3R_{0}^{2}}{16}>0
\]
for $|y|<y_{0}$ and $0<z<z_{0}$.  The selected branch is therefore
\[
x=u_{+}(y,z)
:=
\sqrt{R(z)^{2}-y^{2}}-R_{0}.
\]
Since $|y|/R(z)\leq1/2$, its inherited normal has positive
$x$-component and agrees with \eqref{eq:two-dimensional-graph-normal}.
Using \eqref{eq:3-R-expansion},
\[
u_{+}(y,z)
=
\alpha z
-
\frac{\tau\sqrt{1+\alpha^{2}}}{2}y^{2}
+
O\bigl(z^{3/2}+y^{2}z+y^{4}\bigr).
\]

\begin{figure}[t]
\centering
\begin{tikzpicture}
\begin{groupplot}[
  group style={group size=2 by 1, horizontal sep=1.15cm},
  width=0.405\textwidth,
  height=0.265\textwidth,
  slce axis
]

\nextgroupplot[
  title={The selected branch as a graph},
  xlabel={$z$}, ylabel={$x$},
  xmin=-0.015, xmax=\figGraphZmax+0.015,
  ymin=-0.02, ymax=0.57
]
\addplot[SLCEPost, line width=1.25pt,
  domain=0:\figGraphZmax, samples=200]
  {\figAlpha*x + \figCusp*x^(1.5)};
\addplot[black!55, line width=0.75pt, densely dotted,
  domain=0:\figGraphZmax, samples=2]
  {\figAlpha*x};
\node[slce point] at (axis cs:0,0) {};
\node[slce label, text=SLCEPost, anchor=west]
  at (axis cs:0.18,0.38) {$x=x_{+}(z)$};
\draw[slce arrow, SLCEPost]
  (axis cs:0.185,0.365) --
  (axis cs:0.19,{\figAlpha*0.19+\figCusp*0.19^(1.5)});
\node[slce label, anchor=west] at (axis cs:0.025,0.515)
  {$x_{+}(z)=\alpha z+O(z^{3/2})$};

\nextgroupplot[
  title={Even reflection across $z=0$},
  xlabel={$z$}, ylabel={$x$},
  xmin=-\figGraphZmax-0.015, xmax=\figGraphZmax+0.015,
  ymin=-0.02, ymax=0.57
]
\addplot[SLCESingular, line width=1.25pt,
  domain=-\figGraphZmax:\figGraphZmax, samples=260]
  {\figAlpha*abs(x) + \figCusp*(abs(x))^(1.5)};
\addplot[black!55, line width=0.65pt, densely dashed]
  coordinates {(0,-0.02) (0,0.57)};
\node[slce point] at (axis cs:0,0) {};
\node[slce label, text=SLCESingular, anchor=west]
  at (axis cs:0.12,0.43) {$x(z)=x_{+}(|z|)$};
\node[slce label, anchor=east] at (axis cs:-0.055,0.125)
  {$x_z(0^-)=-\alpha$};
\node[slce label, anchor=west] at (axis cs:0.055,0.125)
  {$x_z(0^+)=\alpha$};

\end{groupplot}
\end{tikzpicture}
\caption{Formation of the Lipschitz corner.  The selected post-focal
branch is first written as a one-sided graph and then reflected across
$z=0$.  The common tangent becomes the two distinct one-sided
slopes $\pm\alpha$.}
\label{fig:lipschitz-reflection}
\end{figure}
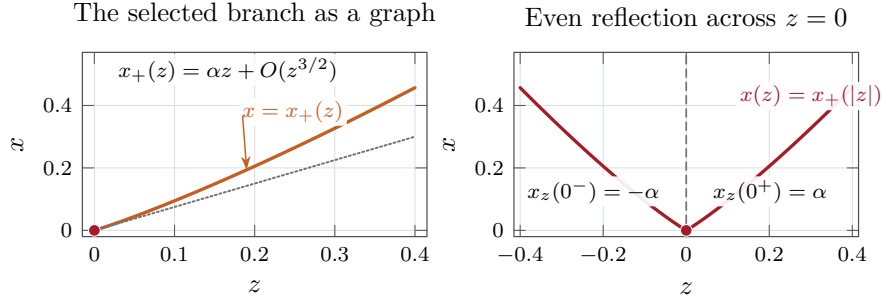

On
\[
Q:=(-y_{0},y_{0})\times(-z_{0},z_{0}),
\]
define
\[
u(y,z):=u_{+}(y,|z|).
\]
The formulas
\[
(u_{+})_{y}
=-\frac{y}{\sqrt{R(z)^{2}-y^{2}}},
\qquad
(u_{+})_{z}
=
\frac{R'(z)R(z)}{\sqrt{R(z)^{2}-y^{2}}}
\]
show that $u\in C^{0,1}(Q)$.  Moreover,
\begin{equation}
\label{eq:3-one-sided-z-derivatives}
u_{z}(y,0^+)
=
\frac{\alpha R_{0}}{\sqrt{R_{0}^{2}-y^{2}}},
\qquad
u_{z}(y,0^-)
=
-\frac{\alpha R_{0}}{\sqrt{R_{0}^{2}-y^{2}}},
\end{equation}
so $u\notin C^{1}(Q)$.  Reflection across $z=0$ is a Euclidean
isometry preserving the chosen graph orientation; hence $u$ is a
classical solution in $Q\setminus\{z=0\}$.  Finally,
\begin{equation}
\label{eq:3-u-final-expansion}
u(y,z)
=
\alpha|z|
-
\frac{\tau\sqrt{1+\alpha^{2}}}{2}y^{2}
+
O\bigl(|z|^{3/2}+y^{2}|z|+y^{4}\bigr).
\end{equation}

\begin{figure}[t]
\centering
\begin{tikzpicture}
\begin{groupplot}[
  group style={group size=2 by 1, horizontal sep=0.85cm},
  width=0.43\textwidth,
  height=0.31\textwidth,
  view={-55}{24},
  xlabel={$y$}, ylabel={$z$}, zlabel={$x$},
  xlabel style={font=\scriptsize},
  ylabel style={font=\scriptsize},
  zlabel style={font=\scriptsize},
  title style={font=\small, align=center},
  tick label style={font=\scriptsize},
  domain=-\figYmax:\figYmax,
  samples=34,
  samples y=26,
  z buffer=sort,
  colormap={slceblue}{
    color(0cm)=(white);
    color(0.55cm)=(SLCEOther!35);
    color(1cm)=(SLCEOther!90!black)
  },
  grid=major,
  major grid style={draw=SLCEGrid, line width=0.22pt},
  axis line style={black!65, line width=0.4pt},
  tick style={black!60},
  clip=false
]

\nextgroupplot[
  title={One-sided graph $x=u_{+}(y,z)$},
  y domain=0:\figSurfaceZmax,
  ymin=0, ymax=\figSurfaceZmax
]
\addplot3[surf, shader=interp, opacity=0.93]
  {sqrt(max((figRPost(y))^2-x^2,0))-\figRzero};
\addplot3[SLCESeed, line width=1.15pt,
  domain=-\figYmax:\figYmax, samples=90]
  (x,0,{sqrt(max(\figRzero*\figRzero-x*x,0))-\figRzero});

\nextgroupplot[
  title={Reflected graph $x=u_{+}(y,|z|)$},
  y domain=-\figSurfaceZmax:\figSurfaceZmax,
  ymin=-\figSurfaceZmax, ymax=\figSurfaceZmax
]
\addplot3[surf, shader=interp, opacity=0.93]
  {sqrt(max((figRPost(abs(y)))^2-x^2,0))-\figRzero};
\addplot3[SLCESingular, line width=1.25pt,
  domain=-\figYmax:\figYmax, samples=90]
  (x,0,{sqrt(max(\figRzero*\figRzero-x*x,0))-\figRzero});

\end{groupplot}
\end{tikzpicture}
\caption{Leading-order local graph geometry.  The left panel is the
smooth one-sided post-focal graph.  After reflection, the two smooth
sheets meet continuously along the red curve $z=0$, where the
normal derivative has a jump.}
\label{fig:lipschitz-surfaces}
\end{figure}
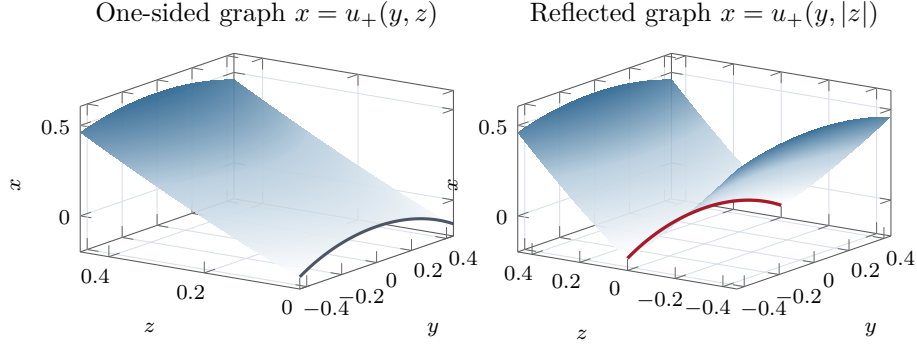

It remains to check the viscosity inequalities on
$\Gamma:=\{(y,0):|y|<y_{0}\}$.  Fix
\[
P_{*}:=(y_{*},0),
\qquad
X_{*}:=\sqrt{R_{0}^{2}-y_{*}^{2}}.
\]
Then
\[
u_{y}(P_{*})=-\frac{y_{*}}{X_{*}},
\qquad
u_{yy}(P_{*})=-\frac{R_{0}^{2}}{X_{*}^{3}},
\]
and \eqref{eq:3-one-sided-z-derivatives} gives
\[
u_{z}(P_{*}^{+})=\frac{\alpha R_{0}}{X_{*}},
\qquad
u_{z}(P_{*}^{-})=-\frac{\alpha R_{0}}{X_{*}}.
\]
If $\varphi\in C^{2}(Q)$ touches $u$ from above at $P_{*}$, restriction
to $y=y_{*}$ gives
\[
u_{z}(P_{*}^{+})
\leq\varphi_{z}(P_{*})
\leq u_{z}(P_{*}^{-}),
\]
a contradiction.  Thus the subsolution condition is vacuous on
$\Gamma$.

If $\varphi$ touches $u$ from below at $P_{*}$, then
\[
\varphi_{y}(P_{*})=-\frac{y_{*}}{X_{*}},
\qquad
\varphi_{yy}(P_{*})\leq-\frac{R_{0}^{2}}{X_{*}^{3}},
\qquad
|\varphi_{z}(P_{*})|\leq\frac{\alpha R_{0}}{X_{*}}.
\]
Consequently,
\[
\sqrt{1+|D\varphi(P_{*})|^{2}}
\leq
\frac{R_{0}\sqrt{1+\alpha^{2}}}{X_{*}}
=
\frac{1}{\tau X_{*}},
\qquad
1+\varphi_{y}(P_{*})^{2}
=
\frac{R_{0}^{2}}{X_{*}^{2}}.
\]
By \eqref{eq:y-normal-curvature}, the normal curvature of the test
graph in the $\partial_{y}$ direction satisfies
\[
k_{y}(P_{*})
\leq
\frac{-R_{0}^{2}/X_{*}^{3}}
{(1/(\tau X_{*}))(R_{0}^{2}/X_{*}^{2})}
=-\tau.
\]
If $k_{1}\leq k_{2}$ are its principal curvatures, then
\eqref{eq:normal-curvature-comparison} gives $k_{1}\leq-\tau$, and
therefore
\[
F_{\Theta}
\bigl(D\varphi(P_{*}),D^{2}\varphi(P_{*})\bigr)
=
\arctan k_{1}+\arctan k_{2}-\Theta
<
\arctan(-\tau)+\frac{\pi}{2}-\Theta
=0.
\]
Thus $u$ is a viscosity supersolution on $\Gamma$, and hence a
viscosity solution in $Q$.  The expansion
\eqref{eq:3-u-final-expansion} and the jump
\eqref{eq:3-one-sided-z-derivatives} finish the proof.
\end{proof}

\begin{remark}[Choice of the constants]
\label{rem:two-dimensional-Lipschitz-bookkeeping}
One nonoptimal explicit choice is
\[
\begin{aligned}
C_{1}
&:=
\frac{1}{\tau\sqrt{K_{0}}\sqrt{1+\alpha^{2}}}
+
\frac{\alpha}{\sqrt{1+\alpha^{2}}},
&
C_{2}
&:=
\frac{1}{\tau\sqrt{1+\alpha^{2}}}
+
\frac{\alpha\sqrt{K_{0}}}{\sqrt{1+\alpha^{2}}},
\\
C_{3}
&:=\frac{3}{\tau},
&
C_{4}
&:=\frac{144K_{0}(1+\alpha^{2})}{\tau^{2}},
\\
C_{5}
&:=8\tau K_{0}^{2}\sqrt{1+\alpha^{2}},
&
C_{6}
&:=4,
\\
C_{7}
&:=\frac{12(1+\alpha^{2})}{\tau},
&
C_{8}
&:=\frac{\alpha K_{0}}
{4\tau\sqrt{1+\alpha^{2}}},
\\
C_{9}
&:=\frac{9\alpha K_{0}}
{4\tau\sqrt{1+\alpha^{2}}},
&
C_{10}
&:=
\left(
\frac{8\tau\sqrt{1+\alpha^{2}}}{\alpha K_{0}}
\right)^{1/2},
\\
C_{11}
&:=
\frac{12(1+\alpha^{2})}{\tau}
\left(
\frac{8\tau\sqrt{1+\alpha^{2}}}{\alpha K_{0}}
\right)^{1/2}.
\end{aligned}
\]
It is enough to take
\[
\begin{aligned}
s_{0}:=\min\Biggl\{
&\frac{1}
{2\tau K_{0}\sqrt{1+\alpha^{2}}\,C_{1}},
\frac{\alpha}
{2\sqrt{1+\alpha^{2}}\,C_{2}},
\frac{1-\alpha/\sqrt{1+\alpha^{2}}}{2C_{2}},
\\
&\frac{1}{2\sqrt{1+\alpha^{2}}\,C_{3}},
\frac{\alpha K_{0}}{2\tau C_{4}},
\frac{K_{0}}{2C_{5}},
\frac{1}{2\tau\sqrt{1+\alpha^{2}}\,C_{6}},
\frac{\alpha}{2C_{7}}
\Biggr\}.
\end{aligned}
\]
These constants follow directly from
\eqref{eq:3-basic-bounds} and the displayed formulas for
$\zeta''$, $J_{s}''$, $J_{\theta}'$, $R'$, and
$(-r'/\zeta')'$.
\end{remark}

\section{Failure of two-dimensional interior curvature estimates}
\label{sec:two-dimensional-C2-failure}

We prove Theorem~\ref{thm:two-dimensional-C2-failure} by choosing a
family of smooth post-focal parallel surfaces for which the meridian
focal factor satisfies $J_{s,\eps}(0)=-\tau\eps$ and converges to a
quadratic zero.  The resulting cubic relation between the meridian
parameter and the graph coordinate gives the uniform
$C^{1,1/3}$ bound and the sharp failure of every higher H\"older
estimate.

\begin{proof}[Proof of
Theorem~\ref{thm:two-dimensional-C2-failure}]
Fix
\[
0<\Theta<\frac{\pi}{2},
\qquad
\tau:=\cot\Theta,
\qquad
K_{0}:=\frac{1}{1+\tau^{2}}.
\]
For every function $f_{\eps}=f_{\eps}(s)$ introduced below, write
$f_{\eps,0}:=f_{\eps}(0)$.  We shrink $s_{0}>0$ whenever necessary,
without changing the notation; one simultaneous choice is given in
Remark~\ref{rem:two-dimensional-C2-bookkeeping}.

For
\[
0\leq\eps\leq\frac{1}{2\tau},
\]
define
\[
r_{\eps}(s)
:=
\frac{1/\tau+\eps}{K_{0}}
\cos\bigl(\sqrt{K_{0}}s\bigr).
\]
Then $r_{\eps}''+K_{0}r_{\eps}=0$ and
\[
r_{\eps,0}
=
\frac{1/\tau+\eps}{K_{0}},
\qquad
r'_{\eps,0}=0,
\qquad
r''_{\eps,0}=-\left(\frac1\tau+\eps\right).
\]
Moreover,
\[
r_{\eps}'''(s)
=
\left(\frac1\tau+\eps\right)
\sqrt{K_{0}}\sin\bigl(\sqrt{K_{0}}s\bigr),
\qquad
|r_{\eps}'''|\leq C_{1},
\qquad
|r_{\eps}''|\leq C_{2}.
\]
After shrinking $s_{0}$,
\[
\frac12\left(\frac1\tau+\eps\right)
\leq-r_{\eps}''
\leq
\frac32\left(\frac1\tau+\eps\right).
\]
Since $r'_{\eps,0}=0$, this gives
\[
sr'_{\eps}(s)\leq0,
\qquad
\frac{|s|}{2\tau}
\leq|r'_{\eps}(s)|
\leq\frac{9|s|}{4\tau}.
\]
Shrinking $s_{0}$ again, Taylor's formula at $s=0$ yields
\begin{equation}
\label{eq:4-basic-bounds}
\frac{r_{\eps,0}}{2}
\leq r_{\eps}\leq
\frac{3r_{\eps,0}}{2},
\qquad
|r'_{\eps}|\leq\frac12.
\end{equation}

Set
\[
\zeta_{\eps}(s)
:=-\int_{0}^{s}\sqrt{1-r'_{\eps}(t)^{2}}\,dt,
\qquad
\zeta'_{\eps}=-\sqrt{1-r_{\eps}'^{\,2}}.
\]
Then
\[
\zeta_{\eps,0}=0,
\qquad
\zeta'_{\eps,0}=-1,
\qquad
r_{\eps}'^{\,2}+\zeta_{\eps}'^{\,2}=1,
\]
and \eqref{eq:4-basic-bounds} gives
\[
\frac12\leq-\zeta'_{\eps}\leq\frac32.
\]
Differentiating the arc-length identity,
\[
\zeta_{\eps}''
=-\frac{r'_{\eps}r''_{\eps}}{\zeta'_{\eps}}
=
\frac{K_{0}r_{\eps}r'_{\eps}}{\zeta'_{\eps}},
\qquad
|\zeta_{\eps}''|\leq C_{2}.
\]

Consider
\[
N_{\eps}(s,\theta)
=
\bigl(r_{\eps}(s)\cos\theta,
      r_{\eps}(s)\sin\theta,
      \zeta_{\eps}(s)\bigr),
\qquad
\nu_{\eps}=-\zeta'_{\eps}e_{r}+r'_{\eps}e_{z}.
\]
A direct computation gives
\[
\lambda_{s,\eps}
=-\frac{r''_{\eps}}{\zeta'_{\eps}}
=
\frac{K_{0}r_{\eps}}{\zeta'_{\eps}},
\qquad
\lambda_{\theta,\eps}
=
\frac{\zeta'_{\eps}}{r_{\eps}},
\qquad
\lambda_{s,\eps}\lambda_{\theta,\eps}=K_{0}.
\]

For the signed parallel surface $M_{\eps}:=N_{\eps}-\tau\nu_{\eps}$,
set
\[
J_{s,\eps}
:=1+\tau\frac{K_{0}r_{\eps}}{\zeta'_{\eps}},
\qquad
J_{\theta,\eps}
:=1+\tau\frac{\zeta'_{\eps}}{r_{\eps}}.
\]
At $s=0$,
\[
J_{s,\eps,0}=-\tau\eps,
\qquad
J_{\theta,\eps,0}
=1-\frac{\tau K_{0}}{1/\tau+\eps}
\geq K_{0}.
\]
The conserved quantity for $r_{\eps}''+K_{0}r_{\eps}=0$ is
\[
r_{\eps}'^{\,2}+K_{0}r_{\eps}^{2}
\equiv
\frac{(1/\tau+\eps)^{2}}{K_{0}},
\qquad
\zeta_{\eps}'^{\,2}-K_{0}r_{\eps}^{2}
=
1-\frac{(1/\tau+\eps)^{2}}{K_{0}}.
\]
Consequently,
\[
J_{s,\eps}'
=
-\tau K_{0}
\left(
\frac{(1/\tau+\eps)^{2}}{K_{0}}-1
\right)
\frac{r'_{\eps}}{(\zeta'_{\eps})^{3}},
\]
\[
J_{s,\eps}''
=
\tau K_{0}^{2}
\left(
\frac{(1/\tau+\eps)^{2}}{K_{0}}-1
\right)
\frac{r_{\eps}\bigl(1+2r_{\eps}'^{\,2}\bigr)}
{(\zeta'_{\eps})^{5}},
\]
and
\[
\begin{aligned}
J_{s,\eps}'''
&=
\tau K_{0}^{2}
\left(
\frac{(1/\tau+\eps)^{2}}{K_{0}}-1
\right)
\frac{r'_{\eps}}{(\zeta'_{\eps})^{7}}
\\
&\quad\times
\left[
1+r_{\eps}'^{\,2}-2r_{\eps}'^{\,4}
-3K_{0}r_{\eps}^{2}
\bigl(3+2r_{\eps}'^{\,2}\bigr)
\right].
\end{aligned}
\]
Thus $J'_{s,\eps,0}=0$ and
\[
J''_{s,\eps,0}
=
-\tau K_{0}
\left(\frac1\tau+\eps\right)
\left(
\frac{(1/\tau+\eps)^{2}}{K_{0}}-1
\right),
\]
so
\[
\frac{K_{0}}{\tau^{2}}
\leq-J''_{s,\eps,0}
\leq\frac{27}{8\tau^{2}}.
\]
The preceding bounds imply $|J_{s,\eps}'''|\leq C_{3}$.  After a
further shrinking,
\[
\frac12\bigl(-J''_{s,\eps,0}\bigr)
\leq-J''_{s,\eps}(s)
\leq
\frac32\bigl(-J''_{s,\eps,0}\bigr).
\]
Integrating twice from $0$ gives
\begin{equation}
\label{eq:4-Js-quadratic-bounds}
\tau\eps+C_{4}s^{2}
\leq-J_{s,\eps}(s)
\leq
\tau\eps+C_{5}s^{2}
\qquad (|s|<s_{0}).
\end{equation}
In particular, $J_{s,\eps}<0$ for $\eps>0$, while
$J_{s,0}(s)<0$ for $s\neq0$.

The same conserved quantity gives
\[
J_{\theta,\eps}'
=
\tau
\left(
\frac{(1/\tau+\eps)^{2}}{K_{0}}-1
\right)
\frac{r'_{\eps}}
{\zeta'_{\eps}r_{\eps}^{2}},
\qquad
|J_{\theta,\eps}'|\leq C_{6}.
\]
Shrinking $s_{0}$ once more,
\[
J_{\theta,\eps}\geq\frac{K_{0}}{2}.
\]
Hence $M_{\eps}$ is regular for $\eps>0$.  In the two principal
directions,
\[
\kappa_{s,\eps}
=
\frac{\lambda_{s,\eps}}{J_{s,\eps}},
\qquad
\kappa_{\theta,\eps}
=
\frac{\lambda_{\theta,\eps}}{J_{\theta,\eps}}.
\]
Since $\lambda_{s,\eps}\lambda_{\theta,\eps}=K_{0}$, direct
substitution gives
\[
K_{M_{\eps}}+\tau H_{M_{\eps}}=1.
\]
Moreover,
\[
\kappa_{s,\eps}+\tau
=
\frac{(1+\tau^{2})J_{s,\eps}-1}
{\tau J_{s,\eps}}>0,
\qquad
\kappa_{\theta,\eps}+\tau
=
\frac{\tau+\lambda_{s,\eps}^{-1}}
{J_{\theta,\eps}}>0.
\]
Indeed, $J_{s,\eps}<0$ implies
$\lambda_{s,\eps}<-1/\tau$ and hence
$\lambda_{s,\eps}^{-1}>-\tau$.  Lemma~\ref{lem:two-dimensional-equivalence}
therefore gives
\[
\arctan\kappa_{s,\eps}
+
\arctan\kappa_{\theta,\eps}
=
\Theta.
\]

\begin{figure}[t]
\centering
\begin{tikzpicture}
\begin{groupplot}[
  group style={group size=3 by 1, horizontal sep=0.62cm},
  width=0.265\textwidth,
  height=0.235\textwidth,
  slce axis,
  title style={font=\scriptsize, align=center}
]

\nextgroupplot[
  title={(a) Focal factor},
  xlabel={$s$}, ylabel={$J_{s,\eps}$},
  xmin=-1.03, xmax=1.03,
  ymin=-1.18, ymax=0.08
]
\addplot[SLCELimit, line width=1.15pt, domain=-1:1, samples=160]
  {-x*x};
\addplot[SLCESmoothA, line width=0.9pt, domain=-1:1, samples=160]
  {-(\figEpsA+x*x)};
\addplot[SLCESmoothB, line width=0.9pt, dashed, domain=-1:1, samples=160]
  {-(\figEpsB+x*x)};
\addplot[SLCESmoothC, line width=0.9pt, dash dot, domain=-1:1, samples=160]
  {-(\figEpsC+x*x)};
\node[slce label, anchor=south east] at (axis cs:0.92,-0.14)
  {$J_{s,0}\asymp-s^2$};
\node[slce label, anchor=north west] at (axis cs:-0.94,-0.38)
  {$\eps\downarrow0$};

\nextgroupplot[
  title={(b) Meridian profiles},
  xlabel={$z$}, ylabel={$x$},
  xmin=-\figLimitZmax, xmax=\figLimitZmax,
  ymin=-0.025, ymax=\figLimitC*1.08
]
\addplot[SLCELimit, line width=1.2pt,
  domain=-\figLimitZmax:\figLimitZmax, samples=220]
  {\figLimitC*(abs(x))^(4/3)};
\addplot[SLCESmoothA, line width=0.9pt,
  domain=-\figLimitZmax:\figLimitZmax, samples=220]
  {\figLimitC*((x*x+(\figEpsA)^3)^(2/3)-(\figEpsA)^2)};
\addplot[SLCESmoothB, line width=0.9pt, dashed,
  domain=-\figLimitZmax:\figLimitZmax, samples=220]
  {\figLimitC*((x*x+(\figEpsB)^3)^(2/3)-(\figEpsB)^2)};
\addplot[SLCESmoothC, line width=0.9pt, dash dot,
  domain=-\figLimitZmax:\figLimitZmax, samples=220]
  {\figLimitC*((x*x+(\figEpsC)^3)^(2/3)-(\figEpsC)^2)};
\node[slce label, anchor=west] at (axis cs:0.08,0.15)
  {$u_{\eps,zz}(0)\asymp\eps^{-1}$};
\node[slce label, anchor=north west] at (axis cs:-0.92,0.48)
  {$\eps\downarrow0$};
\node[slce label, anchor=west] at (axis cs:0.42,0.62)
  {$C|z|^{4/3}$};

\nextgroupplot[
  title={(c) Limit graph},
  xlabel={$y$}, ylabel={$z$}, zlabel={$x$},
  view={-55}{24},
  domain=-\figLimitYmax:\figLimitYmax,
  y domain=-\figLimitZmax:\figLimitZmax,
  samples=25,
  samples y=25,
  z buffer=sort,
  colormap={slcegreen}{
    color(0cm)=(white);
    color(0.55cm)=(SLCESmoothC!35);
    color(1cm)=(SLCESmoothC!90!black)
  }
]
\addplot3[surf, shader=interp, opacity=0.93]
  {\figLimitC*(abs(y))^(4/3)-x*x/(2*\figRhoLimit)};
\addplot3[SLCELimit, line width=1.05pt,
  domain=-\figLimitYmax:\figLimitYmax, samples=80]
  (x,0,{-x*x/(2*\figRhoLimit)});

\end{groupplot}
\end{tikzpicture}
\caption{Schematic of the second-order focal degeneration.  The
limiting focal factor vanishes quadratically, which produces the
cubic relation $z\asymp-s^3$ and hence the critical meridian profile
$x\asymp |z|^{4/3}$.  The smooth profiles in the middle panel have
the same transition scale and curvature blow-up rate as the exact
family.}
\label{fig:curvature-concentration-schematic}
\end{figure}
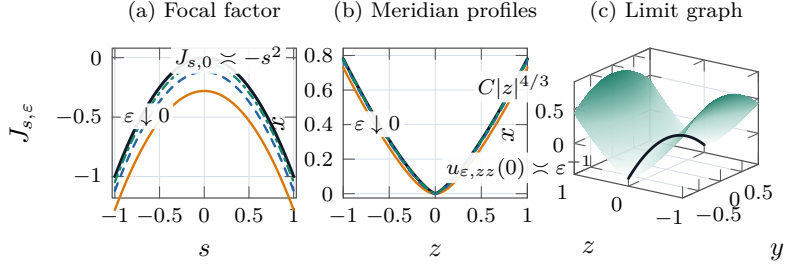

Write
\[
M_{\eps}(s,\theta)
=
\bigl(R_{\eps}(s)\cos\theta,
      R_{\eps}(s)\sin\theta,
      Z_{\eps}(s)\bigr),
\]
where
\[
R_{\eps}:=r_{\eps}+\tau\zeta'_{\eps},
\qquad
Z_{\eps}:=\zeta_{\eps}-\tau r'_{\eps}.
\]
At $s=0$,
\[
R_{\eps,0}
=
\frac1\tau+\frac{\eps}{K_{0}},
\qquad
Z_{\eps,0}=0.
\]
Using $r_{\eps}''=-K_{0}r_{\eps}$ and the displayed formula for
$\zeta_{\eps}''$, we obtain
\begin{equation}
\label{eq:4-RZ-prime}
R_{\eps}'=r'_{\eps}J_{s,\eps},
\qquad
Z_{\eps}'=\zeta'_{\eps}J_{s,\eps}.
\end{equation}
After another shrinking, \eqref{eq:4-basic-bounds} and
\eqref{eq:4-Js-quadratic-bounds} give
$|R_{\eps}'|\leq C_{7}$ and
\begin{equation}
\label{eq:4-R-bounds}
\frac{1}{2\tau}
\leq R_{\eps}(s)
\leq\frac{9}{4\tau K_{0}}.
\end{equation}

Use the graph coordinate
\[
z:=-Z_{\eps}(s).
\]
By \eqref{eq:4-RZ-prime} and
\eqref{eq:4-Js-quadratic-bounds},
\[
Z_{\eps}'(s)
=(-\zeta'_{\eps})(-J_{s,\eps})
\geq C_{8}s^{2}.
\]
Thus $s\mapsto-Z_{\eps}(s)$ is strictly decreasing, and for
$-s_{0}<t<s<s_{0}$,
\begin{equation}
\label{eq:4-cubic-distance}
\bigl|Z_{\eps}(s)-Z_{\eps}(t)\bigr|
\geq
C_{8}\int_{t}^{s}\xi^{2}\,d\xi
\geq
\frac{C_{8}}{12}|s-t|^{3}.
\end{equation}
Set
\[
z_{0}:=\frac{C_{8}}{6}s_{0}^{3}.
\]
Then
\[
(-z_{0},z_{0})
\subset
-Z_{\eps}\bigl((-s_{0},s_{0})\bigr)
\]
for every $0\leq\eps\leq1/(2\tau)$.  Let $s=s_{\eps}(z)$ denote the
inverse on this interval and regard $R_{\eps}$ as a function of $z$.
The identities in \eqref{eq:4-RZ-prime} give
\begin{equation}
\label{eq:4-Rz}
\frac{dR_{\eps}}{dz}
=-\frac{r'_{\eps}}{\zeta'_{\eps}},
\qquad
\left(
-\frac{r'_{\eps}}{\zeta'_{\eps}}
\right)'
=
\frac{K_{0}r_{\eps}}{(\zeta'_{\eps})^{3}}.
\end{equation}
In particular,
\[
\left|\frac{dR_{\eps}}{dz}\right|\leq1,
\qquad
\left|
\left(-\frac{r'_{\eps}}{\zeta'_{\eps}}\right)'
\right|
\leq C_{9}.
\]

Let
\[
y_{0}:=\frac{1}{4\tau},
\qquad
Q:=(-y_{0},y_{0})\times(-z_{0},z_{0}).
\]
By \eqref{eq:4-R-bounds},
\[
R_{\eps}(z)^{2}-y^{2}
\geq\frac{3}{16\tau^{2}}
\qquad ((y,z)\in Q).
\]
For $\eps>0$, the parallel surface is therefore represented on $Q$
by the smooth graph
\begin{equation}
\label{eq:4-u-epsilon}
u_{\eps}(y,z)
:=
\sqrt{R_{\eps}(z)^{2}-y^{2}}-R_{\eps,0}.
\end{equation}
Since $|y|/R_{\eps}(z)\leq1/2$, the inherited normal has positive
$x$-component and agrees with \eqref{eq:two-dimensional-graph-normal}.
Thus $u_{\eps}$ is a classical solution of the SLCE in $Q$.  Moreover,
\[
(u_{\eps})_{y}
=-\frac{y}{\sqrt{R_{\eps}^{2}-y^{2}}},
\qquad
(u_{\eps})_{z}
=
\frac{R_{\eps}'R_{\eps}}
{\sqrt{R_{\eps}^{2}-y^{2}}},
\]
and the preceding bounds imply
\[
\|u_{\eps}\|_{C^{1}(Q)}\leq C,
\]
where $C$ is independent of $\eps$.

At the origin, all derivatives of $R_{\eps}$ in the following displays
are taken with respect to $z$.  We have
\[
R_{\eps}'(0)=0,
\qquad
\frac{d^{2}R_{\eps}}{dz^{2}}
=
-\frac{K_{0}r_{\eps}}
{(\zeta'_{\eps})^{4}J_{s,\eps}}.
\]
Consequently,
\begin{equation}
\label{eq:4-Rzz-origin}
R_{\eps}''(0)
=
\frac{1/\tau+\eps}{\tau\eps}
=
\frac{1}{\tau^{2}\eps}+\frac1\tau.
\end{equation}
Formula \eqref{eq:4-u-epsilon} gives
\[
(u_{\eps})_{zz}(0,0)=R_{\eps}''(0),
\qquad
(u_{\eps})_{yy}(0,0)=-\frac1{R_{\eps,0}},
\qquad
(u_{\eps})_{yz}(0,0)=0.
\]
Hence $|D^{2}u_{\eps}(0,0)|\to\infty$, and every $u_{\eps}$ is
nonconvex near the origin.  Since $Du_{\eps}(0,0)=0$, the shape
operator at the origin equals $D^{2}u_{\eps}(0,0)$, so the curvature
also blows up.

For $z,\widetilde z\in(-z_{0},z_{0})$, let
$s=s_{\eps}(z)$ and $t=s_{\eps}(\widetilde z)$.  Combining
\eqref{eq:4-cubic-distance} and \eqref{eq:4-Rz},
\begin{equation}
\label{eq:4-Rprime-holder}
|R_{\eps}'(z)-R_{\eps}'(\widetilde z)|
\leq
C_{10}|z-\widetilde z|^{1/3}.
\end{equation}
For $\eps=0$, define $u_{0}$ by \eqref{eq:4-u-epsilon} using the
continuous inverse of $-Z_{0}$.  Formula \eqref{eq:4-Rz} defines its
derivative continuously at $z=0$, because $r'_{0,0}=0$.  The displayed
formulas for $Du_{\eps}$, \eqref{eq:4-R-bounds}, and
\eqref{eq:4-Rprime-holder} now give
\[
\sup_{0\leq\eps\leq1/(2\tau)}
\|u_{\eps}\|_{C^{1,1/3}(Q)}
\leq C.
\]

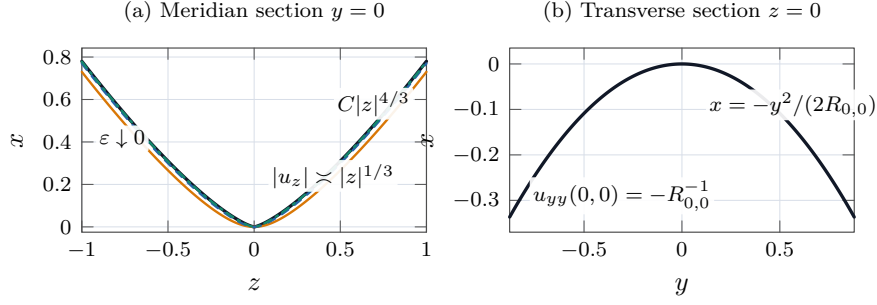
\begin{figure}[t]
\centering
\begin{tikzpicture}
\begin{groupplot}[
  group style={group size=2 by 1, horizontal sep=1.10cm},
  width=0.405\textwidth,
  height=0.265\textwidth,
  slce axis,
  title style={font=\scriptsize, align=center}
]

\nextgroupplot[
  title={(a) Meridian section $y=0$},
  xlabel={$z$}, ylabel={$x$},
  xmin=-\figLimitZmax, xmax=\figLimitZmax,
  ymin=-0.025, ymax=\figLimitC*1.08
]
\addplot[SLCELimit, line width=1.25pt,
  domain=-\figLimitZmax:\figLimitZmax, samples=220]
  {\figLimitC*(abs(x))^(4/3)};
\addplot[SLCESmoothA, line width=0.9pt,
  domain=-\figLimitZmax:\figLimitZmax, samples=220]
  {\figLimitC*((x*x+(\figEpsA)^3)^(2/3)-(\figEpsA)^2)};
\addplot[SLCESmoothB, line width=0.9pt, dashed,
  domain=-\figLimitZmax:\figLimitZmax, samples=220]
  {\figLimitC*((x*x+(\figEpsB)^3)^(2/3)-(\figEpsB)^2)};
\addplot[SLCESmoothC, line width=0.9pt, dash dot,
  domain=-\figLimitZmax:\figLimitZmax, samples=220]
  {\figLimitC*((x*x+(\figEpsC)^3)^(2/3)-(\figEpsC)^2)};
\node[slce label, anchor=west] at (axis cs:0.46,0.58)
  {$C|z|^{4/3}$};
\node[slce label, anchor=north west] at (axis cs:-0.92,0.48)
  {$\eps\downarrow0$};
\node[slce label, anchor=west] at (axis cs:0.08,0.24)
  {$|u_z|\asymp |z|^{1/3}$};

\nextgroupplot[
  title={(b) Transverse section $z=0$},
  xlabel={$y$}, ylabel={$x$},
  xmin=-\figLimitYmax, xmax=\figLimitYmax,
  ymin=-0.37, ymax=0.035
]
\addplot[SLCELimit, line width=1.25pt,
  domain=-\figLimitYmax:\figLimitYmax, samples=180]
  {-x*x/(2*\figRhoLimit)};
\node[slce label, anchor=west] at (axis cs:0.12,-0.09)
  {$x=-y^2/(2R_{0,0})$};
\node[slce label, anchor=west] at (axis cs:-0.78,-0.29)
  {$u_{yy}(0,0)=-R_{0,0}^{-1}$};

\end{groupplot}
\end{tikzpicture}
\caption{The principal sections of the limiting profile.  Along
$y=0$, the derivative is exactly of order $|z|^{1/3}$; along
$z=0$, the profile remains an ordinary quadratic arc.  Thus only
the focal direction carries the critical loss of second-order
regularity.}
\label{fig:critical-one-third-sections}
\end{figure}

Let $0<\eps<\min\{1/(2\tau),s_{0}^{2}\}$ and put
\[
\sigma_{\eps}:=\sqrt\eps,
\qquad
z_{\eps}^{\sharp}:=-Z_{\eps}(\sigma_{\eps}).
\]
For sufficiently small $\eps$, one has
$z_{\eps}^{\sharp}\in(-z_{0},0)$.  The estimate for $-r_{\eps}''$
gives
\[
|R_{\eps}'(z_{\eps}^{\sharp})|
=
\left|
\frac{r'_{\eps}(\sigma_{\eps})}
{\zeta'_{\eps}(\sigma_{\eps})}
\right|
\geq
\frac{1}{2\tau}\eps^{1/2}.
\]
Since $-\zeta'_{\eps}\leq1$, \eqref{eq:4-Js-quadratic-bounds} gives
\[
|z_{\eps}^{\sharp}|
\leq
\int_{0}^{\sigma_{\eps}}
\bigl(\tau\eps+C_{5}s^{2}\bigr)\,ds
=
\left(\tau+\frac{C_{5}}{3}\right)\eps^{3/2}.
\]
As $R_{\eps}'(0)=0$ and
$(u_{\eps})_{z}(0,z)=R_{\eps}'(z)$, for every $\beta>1/3$,
\[
[Du_{\eps}]_{C^{0,\beta}(Q)}
\geq
C_{\beta}\eps^{(1-3\beta)/2}
\longrightarrow\infty.
\]

Since $-Z_{\eps}\to-Z_{0}$ and $R_{\eps}\to R_{0}$ uniformly in the
$s$-parameter, strict monotonicity of $-Z_{0}$ gives
$u_{\eps}\to u_{0}$ uniformly on $Q$.  It remains to identify this
limiting profile.  For $\eps=0$, Taylor expansion at $s=0$ gives
\[
r'_{0}(s)
=-\frac{s}{\tau}
+\frac{K_{0}}{6\tau}s^{3}
+O(s^{5}),
\qquad
\zeta'_{0}(s)
=-1+\frac{s^{2}}{2\tau^{2}}+O(s^{4}),
\]
and hence
\[
J_{s,0}(s)
=-\frac{K_{0}}{2\tau^{2}}s^{2}+O(s^{4}),
\qquad
-Z_{0}(s)
=-\frac{K_{0}}{6\tau^{2}}s^{3}+O(s^{5}).
\]
Writing
\[
a_{\Theta}:=\frac{K_{0}}{6\tau^{2}},
\]
the inverse of $z=-Z_{0}(s)$ satisfies
\[
s_{*}(z)
=-a_{\Theta}^{-1/3}
\sgn(z)|z|^{1/3}+O(|z|).
\]
Moreover,
\[
-\frac{r'_{0}(s)}{\zeta'_{0}(s)}
=-\frac{s}{\tau}+O(s^{3}),
\]
so \eqref{eq:4-Rz} yields
\[
R_{0}'(z)
=
c_{\Theta}\sgn(z)|z|^{1/3}+O(|z|),
\qquad
c_{\Theta}
:=
\left(\frac{6}{\tau K_{0}}\right)^{1/3}.
\]
Since $u_{0}(0,z)=R_{0}(z)-R_{0,0}$,
\[
u_{0}(0,z)
=
\frac{3c_{\Theta}}{4}|z|^{4/3}+O(z^{2}).
\]
Thus $u_{0}$ has exactly the $C^{1,1/3}$ regularity in the focal
direction and belongs to no $C^{1,\beta}$ class with $\beta>1/3$.
This completes the proof.
\end{proof}

\begin{remark}[Choice of the constants]
\label{rem:two-dimensional-C2-bookkeeping}
One nonoptimal explicit choice is
\[
\begin{aligned}
C_{1}&:=\frac{3\sqrt{K_{0}}}{2\tau},
&
C_{2}&:=\frac{3}{2\tau},
&
C_{3}&:=\frac{198K_{0}}{\tau}+\frac{15309}{2\tau^{3}},
\\
C_{4}&:=\frac{K_{0}}{4\tau^{2}},
&
C_{5}&:=\frac{81}{32\tau^{2}},
&
C_{6}&:=9\tau K_{0},
\\
C_{7}&:=\frac12,
&
C_{8}&:=\frac{K_{0}}{8\tau^{2}},
&
C_{9}&:=\frac{18}{\tau},
\\
C_{10}
&:=
\frac{18}{\tau}
\left(\frac{96\tau^{2}}{K_{0}}\right)^{1/3}.
\end{aligned}
\]
It is enough to take
\[
s_{0}
:=
\min\left\{
\frac{1}{2\tau C_{1}},
\frac{2\tau}{9},
\sqrt{\frac{2}{3K_{0}}},
\frac{K_{0}}{2\tau^{2}C_{3}},
\frac{K_{0}}{2C_{6}},
\frac{1}{\sqrt{2C_{5}}},
\frac{1}{2\tau C_{7}}
\right\}.
\]
These constants follow directly from the displayed formulas for
$r_{\eps}'''$, $\zeta_{\eps}''$, $J_{s,\eps}'''$,
$J_{\theta,\eps}'$, $R_{\eps}'$, and
$(-r'_{\eps}/\zeta'_{\eps})'$.
\end{remark}

The geometric scaling behind the preceding construction is summarized
in \cref{fig:curvature-concentration-schematic}.  The model profiles in
its middle panel are chosen so that the transition occurs on the scale
$|z|\asymp\eps^{3/2}$ and their second derivative at the
origin is of order $\eps^{-1}$, exactly as in
\eqref{eq:4-Rzz-origin}.

The two distinguished sections of the limiting graph are shown in
\cref{fig:critical-one-third-sections}.

\begin{remark}\label{rem:one-third-endpoint}
The preceding construction does not violate a $C^{1,1/3}$ estimate. Rather, it
realizes the exponent $1/3$ as the exact borderline for this
post-focal degeneration: the family is uniformly bounded in
$C^{1,1/3}$, while every uniform $C^{1,\beta}$ estimate with
$\beta>1/3$ fails.
\end{remark}

\section{A three-dimensional subcritical Lipschitz solution}
\label{sec:three-dimensional-Lipschitz}

In this section, we prove Theorem~\ref{thm:three-dimensional-Lipschitz}. 
The construction follows the mechanism of Mooney--Savin \cite{MS24}, 
adapting the degenerate Bellman equation to the curvature operator. 

Fix
\[
0<\Theta<\frac{\pi}{2},
\qquad
\Theta^{*}:=\frac{\pi}{2}-\Theta
\in\left(0,\frac{\pi}{2}\right).
\]
For a symmetric matrix $M\in\R^{3\times3}_{\mathrm{sym}}$, set
\[
\mathcal{F}(x,M)
:=
\tr\!\left[
\arctan\!\left(P(x)^{1/2}MP(x)^{1/2}\right)
\right].
\]
Thus $\widetilde{F}_{\Theta^{*}}(x,M) = \mathcal{F}(x,M)-\Theta^{*}$.
We denote the derivatives of $\mathcal{F}$ with respect to the matrix entries by $\mathcal{F}^{ij}(x, M) := \frac{\partial \mathcal{F}}{\partial M_{ij}}(x, M)$. If $\widehat M=P(x)^{1/2}MP(x)^{1/2}$, then
\begin{equation}
\label{eq:5-ellipticity}
\mathcal{F}^{ij}(x,M)M_{ij}
=
\left[
P(x)^{1/2}(I+\widehat M^{2})^{-1}P(x)^{1/2}
\right]_{ij}M_{ij}.
\end{equation}
The coefficient matrix $\mathcal{F}^{ij}$ is positive definite, making $\mathcal{F}$ analytic and strictly elliptic in the Hessian variable.

For $\lambda>0$, we consider
\[
\Phi^{\lambda}(x)
:=
\frac{\lambda x_{1}^{2}}{1+x_{3}}
+
\frac{\lambda x_{2}^{2}}{1-x_{3}},
\qquad |x_{3}|<1.
\]
A direct computation shows that 
$\det D^{2}\Phi^{\lambda}\equiv0$ and $D^{2}\Phi^{\lambda}$ has two positive eigenvalues and one zero 
eigenvalue throughout $\{|x_{3}|<1\}$. Furthermore, 
$\cof\bigl(D^{2}\Phi^{\lambda}\bigr)_{33} = 4\lambda^{2}/(1-x_{3}^{2})>0$.

\begin{lemma}
\label{lem:dual-phase-minimum}
The dual phase $\vartheta_{\lambda}(x) := \mathcal{F}\bigl(x,D^{2}\Phi^{\lambda}(x)\bigr)$ has a non-degenerate local minimum at the origin for all $\lambda>0$.
\end{lemma}
\begin{proof}
Let $S(x) := \tilde{S}[\Phi^\lambda]$ and $f(t) = \arctan t$. Using the expansion
\[
P(x)^{1/2} = I + \frac{1}{4}|x|^2 I + \frac{1}{2}x\otimes x + O(|x|^4)
\]  
and
\[
\Phi^\lambda(x) = \lambda(x_1^2 + x_2^2 + x_3(x_2^2 - x_1^2) + x_3^2(x_1^2 + x_2^2) + O(|x|^5)),
\]
we have $S(0) = \diag(2\lambda, 2\lambda, 0)$. Letting $\mu_{i}$ denote the eigenvalues of $M_{\lambda}(0)$, the standard spectral formulas for the trace of a matrix function yield
\[
\partial_{k}\vartheta_{\lambda} = \sum_{i=1}^{3} f'(\mu_{i})\bigl(\partial_{k}S\bigr)_{ii},
\]
\begin{align*}
\partial_{kl}\vartheta_{\lambda} = & \sum_{i=1}^{3} f'(\mu_{i})\bigl(\partial_{kl}S\bigr)_{ii}
+ \sum_{i=1}^{3} f''(\mu_{i})\bigl(\partial_{k}S\bigr)_{ii}\bigl(\partial_{l}S\bigr)_{ii} \\
& + 2\sum_{i<j} \frac{f'(\mu_{i}) - f'(\mu_{j})}{\mu_{i} - \mu_{j}} \bigl(\partial_{k}S\bigr)_{ij}\bigl(\partial_{l}S\bigr)_{ij}.
\end{align*}
At $x=0$, evaluating the derivatives gives 
$\partial_{1}S(0) = -2\lambda(E_{13}+E_{31})$, 
$\partial_{2}S(0) = 2\lambda(E_{23}+E_{32})$, 
and $\partial_{3}S(0) = 2\lambda(-E_{11}+E_{22})$,
where the matrix unit $E_{ij}$ is defined entry-wise as $(E_{ij})_{kl} = \delta_{ik}\delta_{jl}$.
Substituting these into the first variation, we find $\partial_{k}\vartheta_{\lambda}(0) = 0$ for all $k$, hence 
\[
D\vartheta_{\lambda}(0) = 0.
\]

Evaluating 
$\partial_{11}S(0) = 2\lambda\diag(3,1,2)$,  
$\partial_{22}S(0) = 2\lambda\diag(1,3,2)$, and 
$\partial_{33}S(0) = 2\lambda\diag(3,3,0)$, and 
using $f'(2\lambda)=(1+4\lambda^2)^{-1}$ and $f''(2\lambda)=-4\lambda(1+4\lambda^2)^{-2}$, 
the spectral formula reduces directly to
\begin{equation}
\label{eq:5-phase-Hessian}
D^{2}\vartheta_{\lambda}(0)
=
\diag\!\left(
\frac{12\lambda}{1+4\lambda^{2}},
\frac{12\lambda}{1+4\lambda^{2}},
\frac{12\lambda+16\lambda^{3}}{(1+4\lambda^{2})^{2}}
\right)>0.
\end{equation}
The mixed entries vanish because $\vartheta_{\lambda}$ is even in $x_{1}$ and $x_{2}$ separately. Thus $\vartheta_{\lambda}$ is strictly convex near the origin.
\end{proof}

Let $\lambda_{0} := \frac12\tan(\Theta^{*}/2)$, so that $\vartheta_{\lambda_{0}}(0)=\Theta^{*}$. Lemma~\ref{lem:dual-phase-minimum} implies that for $0 < \lambda < \lambda_{0}$ sufficiently close to $\lambda_{0}$, the connected component $K$ of $\{\vartheta_{\lambda}\leq\Theta^{*}\}$ containing the origin is a compact, strictly convex analytic body. On $K$, we have $\det D^{2}\Phi^{\lambda}=0$ and $\mathcal{F}(x,D^{2}\Phi^{\lambda})\leq\Theta^{*}$, with strict phase inequality in $\Int K$.

\begin{lemma}
\label{lem:exterior-cauchy-determinant}
There exists an analytic solution $v$ to the Cauchy problem
\[
\begin{cases}
\mathcal{F}(x,D^{2}v)=\Theta^{*},
&\text{outside }K,
\\
v=\Phi^{\lambda},\quad v_{\nu}=\Phi^{\lambda}_{\nu},
&\text{on }\partial K,
\end{cases}
\]
defined on a uniform exterior collar of $\partial K$ and $\nu$ denotes the exterior unit normal to $\partial K$.. Moreover, $\det D^{2}v<0$ on this collar.
\end{lemma}
\begin{proof}
By \eqref{eq:5-ellipticity}, $\mathcal{F}^{ij}\nu_{i}\nu_{j} > 0$, ensuring the boundary is non-characteristic. The Cauchy--Kovalevskaya theorem then yields a unique analytic solution $v$ on an exterior collar.
Because $v$ and $\Phi^{\lambda}$ share the same Cauchy data and solve the same equation on $\partial K$, their full Hessians agree, so $D^{2}v=D^{2}\Phi^{\lambda}$ on $\partial K$.
Evaluating the normal derivative of the phase equation gives
$0 < \partial_{\nu}\vartheta_{\lambda} = -\mathcal{F}^{\nu\nu}(v_{\nu\nu\nu}-\Phi^{\lambda}_{\nu\nu\nu})$,
which implies $\omega := v_{\nu\nu\nu} - \Phi^{\lambda}_{\nu\nu\nu} < 0$ on $\partial K$.

To determine the sign of the determinant on the collar, let $G(M) = \det M$. We compute the normal derivatives of $G(D^2v)$ at a point $p\in \partial K$. Since $G(D^2v) = G(D^2\Phi^\lambda) = 0$ on the boundary, we have
\[
\partial_{\nu}(G(D^{2}v)) = G_{ij}(v_{ij\nu} - \Phi^{\lambda}_{ij\nu}) = G_{\nu\nu}(v_{\nu\nu\nu} - \Phi^{\lambda}_{\nu\nu\nu}) = G_{\nu\nu}\omega.
\]
Let $\xi$ span the kernel of $D^{2}\Phi^{\lambda}(p)$. The cofactor matrix is given by $G_{ij} = \gamma\xi_{i}\xi_{j}$ for some $\gamma>0$.
If $\xi$ is not tangent to $\partial K$, then $G_{\nu\nu} = \gamma(\xi\cdot\nu)^{2} > 0$, yielding $\partial_{\nu}(G(D^{2}v)) < 0$.
If $\xi$ is tangent to $\partial K$, then $G_{\nu\nu} = 0$, implying $\partial_{\nu}(G(D^{2}v)) = 0$. In this case, we have $G_{\xi\xi} v_{\xi\xi\nu} = 0$, giving $v_{\xi\xi\nu} = \Phi^{\lambda}_{\xi\xi\nu} = 0$. We then compute the second normal derivative:
\[
\partial_{\nu\nu}(G(D^{2}v)) = G_{\xi\xi}(v_{\xi\xi\nu\nu} - \Phi^{\lambda}_{\xi\xi\nu\nu}) + G_{ij,kl}(v_{ij\nu}v_{kl\nu} - \Phi^{\lambda}_{ij\nu}\Phi^{\lambda}_{kl\nu}) = I + II.
\]
Since all third-order derivatives of $v$ and $\Phi^{\lambda}$ involving a tangential direction agree, the only possible non-zero terms in $II$ are those with $i=j=\nu$ or $k=l=\nu$. However, $G_{\nu\nu,kl} = 0$ unless $(k,l) = (\xi,\xi)$. Thus, the term $II$ vanishes.
To evaluate $I$, we differentiate $v_{\nu\nu} - \Phi^{\lambda}_{\nu\nu} = 0$ twice along the boundary geodesic in direction $\xi$ to obtain $(v_{\nu\nu} - \Phi^{\lambda}_{\nu\nu})_{\xi\xi} = \kappa_{\xi}(v_{\nu\nu\nu} - \Phi^{\lambda}_{\nu\nu\nu})$, where $\kappa_{\xi}>0$ is the normal curvature of $\partial K$.
Consequently, $\partial_{\nu\nu}(G(D^{2}v)) = \gamma\kappa_{\xi}\omega < 0$. By compactness, $\det D^2v < 0$ on a uniform exterior collar.
\end{proof}

We define $w\in C^{2,1}$ by piecing together $\Phi^\lambda$ and $v$:
\[
w(x)
:=
\begin{cases}
\Phi^{\lambda}(x),&x\in K,
\\
v(x),&x\text{ in the exterior collar}.
\end{cases}
\]
This yields a solution to the degenerate Bellman equation
$\max\left\{\widetilde{F}_{\Theta^{*}}(x,D^{2}w), \det D^{2}w\right\}=0$.
By continuity, we may assume the collar is small enough that $D^{2}w$ retains the inertia $(+,+,-)$ outside $K$, and $\cof(D^{2}w)_{33}>0$ globally.

\begin{lemma}
\label{lem:gradient-collapse-injectivity}
For a sufficiently small neighborhood $U_{0}$ of $K$, the gradient map $\nabla w$ collapses vertical fibres of a transformed convex body onto an analytic surface, and is a global diffeomorphism outside this surface.
\end{lemma}
\begin{proof}
Consider the coordinate map $H(x):=(w_{1}(x),w_{2}(x),x_{3})$. The Jacobian matrix of $H$ is
\[
DH = \begin{pmatrix} w_{11} & w_{12} & w_{13} \\ w_{21} & w_{22} & w_{23} \\ 0 & 0 & 1 \end{pmatrix},
\]
which gives $\det DH = w_{11}w_{22} - w_{12}^{2} = \cof(D^{2}w)_{33} > 0$.
By taking $U_{0}$ sufficiently small, $D^{2}w$ is uniformly close to $2\lambda(I-e_{3}\otimes e_{3})$, making $DH$ uniformly close to $\diag(2\lambda, 2\lambda, 1)$. Thus, $H$ is a distance-expanding global diffeomorphism onto its image, and $H(K)$ is a strictly convex analytic body.

In the domain $D := H(U_0)$, we define the transformed gradient map $T := \nabla w\circ H^{-1}$. Given the structure of $H$, we can explicitly write $T(z) = (z_{1}, z_{2}, w_{3}(H^{-1}(z)))$. The Jacobian matrix $DT(z)$ is therefore lower triangular:
\[
DT(z) = \begin{pmatrix} 1 & 0 & 0 \\ 0 & 1 & 0 \\ \partial_{z_{1}}T_{3} & \partial_{z_{2}}T_{3} & \partial_{z_{3}}T_{3} \end{pmatrix}.
\]
Consequently, its determinant is exactly the vertical derivative: $\det DT(z) = \partial_{z_{3}}T_{3}(z)$.
On the other hand, the chain rule gives $DT(z) = D^{2}w(H^{-1}(z))\,DH^{-1}(z)$, which yields the identity
\[
\partial_{z_{3}}T_{3}(z) = \det DT(z) = \frac{\det D^{2}w(H^{-1}(z))}{\det DH(H^{-1}(z))}.
\]
By Lemma~\ref{lem:exterior-cauchy-determinant}, $\partial_{z_{3}}T_{3} = 0$ on $H(K)$ and $\partial_{z_{3}}T_{3} < 0$ on $D\setminus H(K)$. Since $H(K)$ is convex, its intersection with any vertical line is a connected segment. The strict monotonicity outside $H(K)$ implies $T_{3}$ is constant on vertical segments inside $H(K)$ and strictly decreasing outside. Thus, $T$ collapses each vertical segment of $H(K)$ to a single point on a compact analytic surface $\Gamma_{0}$, and is injective on $D\setminus H(K)$.
\end{proof}

\begin{figure}[t]
\centering
\begin{tikzpicture}[
    x=0.94cm,
    y=0.94cm,
    >=Latex,
    every node/.style={font=\small},
    map arrow/.style={-{Stealth[length=2.1mm,width=1.45mm]},
      line width=0.75pt, draw=black!78},
    outer domain/.style={draw=SLCESeed!88, line width=0.72pt,
      fill=SLCEGrid!28},
    bellman core/.style={draw=SLCESmoothB!75!black, line width=0.78pt,
      fill=SLCESmoothB!17},
    collapsed surface/.style={draw=SLCESingular!78!black,
      line width=0.78pt, fill=SLCESingular!13},
    guide line/.style={densely dashed, line width=0.65pt,
      draw=black!58},
    fibre/.style={line width=1.45pt, draw=SLCESmoothB!80!black},
    panel label/.style={font=\small, align=center},
    object label/.style={font=\small, fill=white, fill opacity=0.88,
      text opacity=1, rounded corners=0.5pt, inner sep=1.2pt}
]

\begin{scope}[shift={(0,0)}]
    \draw[outer domain]
        plot[smooth cycle, tension=0.75] coordinates {
            (-1.45,0.20) (-1.00,1.10) (0.05,1.35)
            (1.05,0.98) (1.42,0.10) (1.05,-0.85)
            (0.02,-1.20) (-1.10,-0.82)
        };

    \draw[bellman core]
        plot[smooth cycle, tension=0.85] coordinates {
            (-0.70,0.10) (-0.42,0.62) (0.10,0.72)
            (0.62,0.42) (0.70,-0.10) (0.42,-0.55)
            (-0.12,-0.65) (-0.58,-0.35)
        };

    \node[object label] at (0,0.02) {$K$};
    \node[object label] at (0.88,0.87) {$H^{-1}(D)$};
    \node[panel label] at (0,-1.55)
        {\textup{(a)} $x=(x_{1},x_{2},x_{3})$};
\end{scope}

\draw[map arrow] (1.70,0) -- (3.05,0)
    node[midway, above=2pt] {$H$};

\begin{scope}[shift={(4.65,0)}]
    \draw[outer domain]
        plot[smooth cycle, tension=0.75] coordinates {
            (-1.50,0.20) (-1.10,1.15) (0.05,1.38)
            (1.15,1.03) (1.55,0.05) (1.12,-0.98)
            (0.00,-1.28) (-1.18,-0.88)
        };

    \draw[bellman core]
        plot[smooth cycle, tension=0.85] coordinates {
            (-0.78,0.18) (-0.45,0.75) (0.12,0.85)
            (0.72,0.48) (0.82,-0.16) (0.45,-0.72)
            (-0.15,-0.82) (-0.70,-0.40)
        };

    \draw[guide line] (0.25,-1.15) -- (0.25,1.20);
    \draw[fibre] (0.25,-0.55) -- (0.25,0.55);
    \fill[SLCESmoothB!82!black] (0.25,-0.55) circle (1.35pt);
    \fill[SLCESmoothB!82!black] (0.25,0.55) circle (1.35pt);

    \draw[decorate,
      decoration={brace, amplitude=4pt, mirror},
      line width=0.58pt, draw=black!72]
        (0.47,0.55) -- (0.47,-0.55)
        node[midway, right=5pt, object label] {$L(z')$};

    \node[object label] at (-0.20,0.05) {$H(K)$};
    \node[object label] at (1.04,0.88) {$D$};
    \node[object label] at (0.25,1.42) {vertical fibre};
    \node[panel label] at (0,-1.55)
        {\textup{(b)} $z=(w_{1},w_{2},x_{3})$};
\end{scope}

\draw[map arrow] (6.45,0) -- (7.95,0)
    node[midway, above=2pt] {$T=\nabla w\circ H^{-1}$};

\begin{scope}[shift={(9.65,0)}]
    \draw[outer domain]
        plot[smooth cycle, tension=0.75] coordinates {
            (-1.35,0.10) (-0.92,1.02) (0.05,1.28)
            (1.05,0.98) (1.35,0.05) (0.98,-0.90)
            (0.00,-1.18) (-1.05,-0.78)
        };

    \draw[collapsed surface]
        plot[smooth, tension=0.9] coordinates {
            (-0.72,-0.08) (-0.35,0.18) (0.12,0.25)
            (0.58,0.05) (0.82,-0.25)
        }
        --
        plot[smooth, tension=0.9] coordinates {
            (0.78,-0.42) (0.35,-0.18) (-0.12,-0.22)
            (-0.58,-0.42) (-0.78,-0.25)
        }
        -- cycle;

    \fill[SLCESingular!85!black] (0.12,0.02) circle (1.55pt);
    \node[object label, right=2pt] at (0.12,0.02) {$y$};

    \draw[map arrow, draw=SLCESingular!82!black]
        (0.12,0.02) -- (0.12,0.75)
        node[above=1pt, object label] {$\partial_{3}u^{-}$};
    \draw[map arrow, draw=SLCESingular!82!black]
        (0.12,0.02) -- (0.12,-0.75)
        node[below=1pt, object label] {$\partial_{3}u^{+}$};

    \node[object label] at (0.82,0.54) {$\Omega$};
    \node[object label] at (-0.49,-0.10) {$\Gamma$};
    \node[panel label] at (0,-1.55)
        {\textup{(c)} $y=(w_{1},w_{2},w_{3})$};
\end{scope}

\end{tikzpicture}
\caption{The collapse underlying the three-dimensional example.
The map $H$ sends the Bellman core $K$ to the convex body
$H(K)$.  For each $z'=(z_{1},z_{2})\in\Int{}E$,
the map $T=\nabla w\circ H^{-1}$ collapses the vertical fibre of
$H(K)$ of length $L(z')$ to one point of $\Gamma$.  After
Legendre transformation, the two endpoints of that fibre become the
one-sided gradients of $u$, and
$\partial_{3}u^{-}-\partial_{3}u^{+}=L(z')$.  The drawing is
schematic and not to scale.}
\label{fig:bellman-gradient-collapse}
\end{figure}
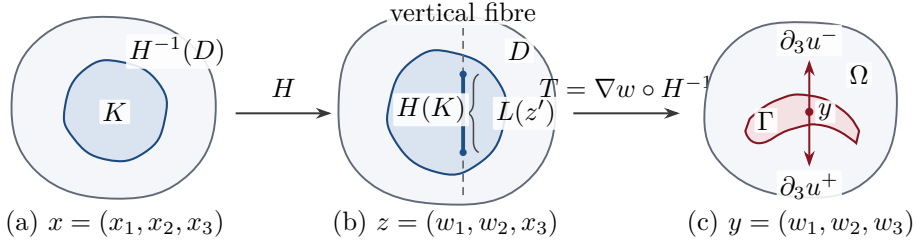

\begin{proof}[Proof of Theorem~\ref{thm:three-dimensional-Lipschitz}]
Let $\Gamma = \Gamma_0$ and choose a smooth bounded domain $\Omega$ such that $\Gamma \Subset \Omega \Subset T(D)$. For $y\in\Omega\setminus\Gamma$, there is a unique $x\in H^{-1}(D)\setminus K$ with $\nabla w(x)=y$. We define the Legendre transform
\[
u(y):=x\cdot y-w(x).
\]
Since $x\cdot y - w(x)$ is constant along any collapsed fibre in $K$, $u$ extends uniquely and continuously to $\Gamma$.
Away from $\Gamma$, $D^{2}w$ has inertia $(+,+,-)$, so Lemma~\ref{lem:Legendre-transform-identity} implies $\tr(\arctan S[u](y)) = \pi/2 - \mathcal{F}(x, D^{2}w(x)) = \Theta$. Thus $u$ is a classical SLCE solution in $\Omega\setminus\Gamma$.

For $z' = (z_1, z_2)$ in the projection of $H(K)$, let $L(z')$ denote the length of the vertical fibre in $H(K)$ above $z'$. As $T_3$ strictly decreases outside $H(K)$, the inverse points from the upper and lower sides of $\Gamma$ approach the endpoints of this fibre. Consequently, the vertical gradient component jumps by
\begin{equation}
\label{eq:5-gradient-jump}
\partial_{3}u^{+}-\partial_{3}u^{-}
=-L(z')<0
\end{equation}
on the relative interior of $\Gamma$. Thus $u$ is Lipschitz but not $C^{1}$.

It remains to verify the viscosity properties on $\Gamma$.
Because $u$ is concave along vertical lines crossing $\Gamma$ (a direct consequence of $T_3$ being strictly decreasing and the orientation of the gradient jump), no $C^2$ function can touch $u$ from below at any point on $\Gamma$. The supersolution condition is therefore vacuous.

To show $u$ is a subsolution, we perturb the dual potential by setting $w_{k}(x) := w(x) - x_{3}^{2}/k$.
By the strict ellipticity of $\mathcal{F}^{ij}$ evaluated on the rank-one segment, $\mathcal{F}(x, D^{2}w_{k}) < \mathcal{F}(x, D^{2}w) \le \Theta^{*}$.
For large $k$, $D^{2}w_{k}$ maintains the inertia $(+,+,-)$, and its gradient map $\nabla w_{k}$ is globally injective on $H^{-1}(D)$ because the corresponding vertical derivative in $H$-coordinates becomes $(\det D^2 w / \det DH) - 2/k < 0$.
The single-valued Legendre transform $u_{k} = w_{k}^{*}$ is of class $C^{2}$ and satisfies
\[
F_{\Theta}(Du_{k},D^{2}u_{k})
=
\Theta^{*}-\mathcal{F}(x,D^{2}w_{k})
>0.
\]
As $w_k \to w$ uniformly, the gradients of $u_k$ remain bounded in the fixed set $H^{-1}(D)$, yielding $u_{k} \to u$ locally uniformly in $\Omega$. The stability of viscosity subsolutions concludes the proof.
\end{proof}

\bibliographystyle{amsplain}
\providecommand{\bysame}{\leavevmode\hbox to3em{\hrulefill}\thinspace}
\providecommand{\MR}{\relax\ifhmode\unskip\space\fi MR }
\providecommand{\MRhref}[2]{%
  \href{http://www.ams.org/mathscinet-getitem?mr=#1}{#2}
}
\providecommand{\href}[2]{#2}

\end{document}